\documentclass[12pt,twoside, leqno]{amsart}
\usepackage{amsmath,amssymb,amsfonts, amscd,   
pifont, amsthm,  amsxtra, latexsym, 
}
\usepackage{comment}
\usepackage{color}
\usepackage[all]{xy}
\usepackage{here}
\usepackage[dvips]{graphicx}

\usepackage{psfrag}
\usepackage[colorlinks=true,linkcolor=blue]{hyperref}
\theoremstyle{plain}
\newtheorem{thm}{Theorem}[section]
\newtheorem{theorem}[thm]{Theorem}

\newtheorem{corollary}[thm]{Corollary}
\newtheorem{proposition}[thm]{Proposition}

\newtheorem{lemma}[thm]{Lemma}

\theoremstyle{definition}
\newtheorem{dfn}[thm]{Definition}

\newtheorem{example}[thm]{Example}
\newtheorem{setting}[thm]{Setting}

\newtheorem{question}[thm]{Question}
\theoremstyle{remark} %
  \newtheorem{remark}[thm]{Remark}%
\makeatletter
\@addtoreset{equation}{section}
\def\@seccntformat#1{\csname the#1\endcsname.\quad}
\def\@seccntformat#1{\csname the#1\endcsname%
\expandafter\ifx\csname#1\endcsname\subsection.\fi\quad}
\makeatother
\newlength{\barlength}
\newcommand{\minusfill}{$\mathsurround=0pt\mathord- \mkern-6mu
   \cleaders\hbox{$\mkern-3mu \mathord- \mkern-3mu$}\hfill
     \mkern-6mu \mathord-$}
\newcommand{\yokobo}{\hbox to 1.2em{\minusfill}}
\newcommand{\equalfill}{$\mathsurround=0pt\mathord= \mkern-6mu
    \cleaders\hbox{$\mkern-3mu \mathord= \mkern-3mu$}\hfill
       \mkern-6mu \mathord=$}
\newcommand{\Longlongrightarrow}
       {\hbox to 2em{\equalfill$\mkern-3mu\Rightarrow$}}
\newcommand{\Longlongleftarrow}
       {\hbox to 2em{$\Leftarrow\mkern-3mu$\equalfill}}
\newcommand{\tsume}{\kern-.35em}
\newlength{\circlength}
\divide\circlength by 2
\advance\circlength by 0.1em

\newcommand{\invHom}[3]{\operatorname{Hom}_{#1}(#2,#3)}

\newcommand{\Kirredrep}[2]{F^{{#1}}({#2})}
\newcommand{\op}{\oplus}
\newcommand{\oi}{\ominus}

\begin{document}
\title[Restriction of $A_{\mathfrak{q}}(\lambda)$ for
 $(GL(n,{\mathbb{R}}), GL(n-1,{\mathbb{R}}))$]
 {How does the restriction of representations change under translations? A story for the general linear groups and the unitary groups }
\author{Toshiyuki Kobayashi and Birgit Speh}


\address[Toshiyuki KOBAYASHI]{%
Graduate School of Mathematical Sciences,
The University of Tokyo;
French-Japanese Laboratory in Mathematics and its Interactions,
FJ-LMI CNRS IRL2025, 
Japan.}

\address[Birgit SPEH]{%
Department of Mathematics,
Cornell University, USA.
}
\maketitle

\bigskip
\begin{center}
{{\textit{Dedicated to Harish-Chandra whose pioneering work is a great inspiration for us}}}
\end{center}

\bigskip
\begin{abstract} We present a new approach to symmetry breaking for pairs of real forms of $(GL(n, \mathbb{C}), GL(n-1, \mathbb{C}))$. 
Translation functors are powerful tools for studying families of representations of a single reductive group $G$.
However, when applied to a pair of groups $G \supset G'$, they can significantly alter the nature of symmetry breaking between the representations of $G$ and $G'$, even within the same Weyl chamber of the direct product group $G \times G'$. 

We introduce the concept of \emph{fences for the interleaving pattern}, which provides a refinement of the usual notion of \emph{walls of Weyl chambers}. We then establish a theorem stating that the multiplicity remains constant unless these \emph{fences} are crossed, together with a new general vanishing theorem for symmetry breaking.

 These general results are illustrated with examples involving both tempered and non-tempered representations. In addition, we present a new non-vanishing theorem for period integrals for pairs of reductive symmetric spaces, which is further strengthened by this approach. \end{abstract}
\medskip



\newpage

\tableofcontents

\medskip

\section{Introduction}
Any finite-dimensional representation $\Pi$ of a compact Lie group $G$ decomposes into a direct sum of irreducible representations when restricted to a subgroup $G'$ of $G$.
A classical result by H.~Weyl (1946) shows that
there is an interlacing pattern between the highest weights of the irreducible summands of $\Pi|_{G'}$  and of the highest weight of $\Pi$ itself. Fix an irreducible representation $\pi$ of $G'$ and  consider the dimension of $\operatorname{Hom}_{G'}(\Pi|_{G'},\pi)$ as a function of the  highest weight of $\Pi$. This function (\emph{multiplicity}) takes only the values 0 and 1 and we can read off the value from the interlacing pattern of the highest weights. In this article, we provide a new proof of the classical result of H.~Weyl  in Section~\ref{sec:upq} and  describe analogous results for certain infinite-dimensional representations of non-compact  Lie groups, which are real forms $GL(n,\mathbb{C})$. 

In contrast to representations of compact Lie groups, 
 the restriction of an irreducible admissible representation of a reductive Lie group
 to a {\it{non-compact}} subgroup
 $G'$ is generally {\it{not}} a direct sum
of irreducible representations. 
Instead of directly decomposing,
 it is useful to consider \emph{symmetry breaking operators} (SBOs),
 which are continuous $G'$-homomorphisms from a topological $G$-module to a topological $G'$-module.  
In this article, we are concerned mainly with the category $\mathcal{M}(G)$
 of admissible smooth representations of $G$
 of finite length having moderate growth, 
 which are defined on topological Fr{\'e}chet  vector spaces
 \cite[Chap.\ 11]{WaI}. 
Let $\operatorname{Irr}(G)$ denote the set of irreducible objects in ${\mathcal{M}}(G)$.  

We denote by 
\begin{equation}
\label{eqn:Pipi} 
   \text{Hom}_{G'}(\Pi|_{G'}, \pi) 
\end{equation}
the space of SBOs, 
 that is, 
 $G'$-homomorphisms from $\Pi \in {\mathcal{M}}(G)$
 to $\pi \in {\mathcal{M}}(G')$, 
 where the operators are continuous in the corresponding Fr\'echet topology.
The dimension of \eqref{eqn:Pipi} is referred
 to as the {\it{multiplicity}}, 
 which we denote by $[\Pi|_{G'}:\pi]$.  

Explicit results on symmetry breaking and multiplicities  for individual non-tempered representations are still sparse.  
For recent works, see \cite{xk21, xksbonvecI, xksbonvec, OS21} for example.  
If both $G$ and $G'$ are classical linear reductive Lie groups with complexified Lie algebras $({\mathfrak{g}}_{\mathbb{C}}, {\mathfrak{g}}_{\mathbb{C}}')=({\mathfrak{gl}}_{n+1}, {\mathfrak{gl}}_n)$ or $({\mathfrak{so}}_{n+1}, {\mathfrak{so}}_n)$,
 and they satisfy Harish-Chandra's rank conditions, 
 then the GGP conjectures/theorems are mostly concerned with non-zero symmetry breaking for $L$-packets or Vogan-packets of discrete series representations \cite{GGP}.

\medskip
For a pair of representations of groups $(G,G')$, 
 where 
\[
   G=GL(n,\mathbb{R}), \ \  G'= GL(n-1,\mathbb{R}), 
\]
 the dimension of the space of symmetry breaking operators is at most one \cite{xsunzhu}.  
 In this article, 
 we introduce a new approach to detecting the non-vanishing of SBOs between irreducible representations that are not necessarily tempered,
 along with a vanishing result and several new non-vanishing results.  
 
In Section~\ref{sec:tfsbo},
 we also introduce the notion of \lq\lq{fences}\rq\rq\ for interlacing patterns, 
 in contrast to the usual concept of \lq\lq{walls}\rq\rq\
 for Weyl chambers.  
While translation functors can significantly alter the nature of symmetry breaking even inside the Weyl chamber (see Example~\ref{ex:alter_mult} for instance), 
 the concept of \lq\lq{fences}\rq\rq\ plays a crucial role in understanding the behavior of symmetry breaking under translations.  

Building on the key results of Theorems~\ref{thm:23081404} and~\ref{thm:24102601}, 
we prove a new \emph{stability theorem} for the multiplicities in symmetry breaking (Theorem \ref{thm:24120703}). 
This theorem applies to general irreducible representations of $G$ and $G'$, which are not necessarily tempered or even unitary.
Here, $(G, G')$ denotes an arbitrary real form of the pair \linebreak $(GL(n,\mathbb{C}), GL(n-1,\mathbb{C}))$, and the theorem asserts that the multiplicity remains constant unless one crosses \lq\lq{fences}\rq\rq. 

In addition, we establish another new result: a general \emph{vanishing theorem} for symmetry breaking (Theorem~\ref{thm:25080811}), formulated in terms of $\tau$-invariants of irreducible representations.

In Section~\ref{sec:upq}, we illustrate these theorems through known examples of symmetry breaking, focusing on tempered representations, 
such as Weyl's branching laws for finite dimensional representations of \linebreak $(U(n), U(n-1))$ and the Gan--Gross--Prasad conjecture for discrete series representations of the pair $(U(p,q), U(p-1,q))$.

In general, proving non-vanishing of symmetry breaking is a difficult problem. 
However, thanks to the stability theorem for multiplicities within fences
(Theorem~\ref{thm:24120703}), 
it suffices to consider representations specified by particular parameters $(\lambda, \nu)$ within the fences.  

In Section~\ref{sec:Spehrep}, 
 we apply this approach to the branching of special unitary representations of $GL(2m,\mathbb{R})$ to the subgroups $GL(2m-1,\mathbb{R})$. In this case, we also obtain non-zero multiplicities for some non-unitary representations.

In Section~\ref{sec:period}, 
we develop a method to detect the existence of a non-zero symmetry breaking operator 
using period integrals for reductive symmetric spaces. 
Theorem~\ref{thm:24012120} provides a new non-vanishing theorem of period integrals related to discrete series representations of a
pair of reductive symmetric spaces.
These results hold for representations that are not necessarily tempered; this is illustrated by examples in Sections~\ref{sec:GHgeneral} and \ref{sec:upq2}.
 
In Section~\ref{sec:GHgeneral}, 
 we discuss symmetry breaking between irreducible representations in the discrete spectrum of
\[
L^2(G/H)=L^2\big(GL(n,\mathbb{R})\big/(GL(p,\mathbb{R})\times GL(n-p,\mathbb{R}))\big)\] and of 
 \[
 L^2(G'/H')=L^2\big(GL(n-1,\mathbb{R})\big/(GL(p,\mathbb{R}) \times GL(n-p-1,\mathbb{R}))\big).\]
These representations are not tempered if $2p<n-1$.  

For this analysis, we examine the phenomenon of \lq\lq{jumping fences}\rq\rq\ in Section \ref{subsec:jump}, in addition to the non-vanishing theorem of period integrals (Theorem~\ref{thm:24012120}).

In Section \ref{sec:upq2}, 
 we discuss symmetry breaking between the irreducible representations in the discrete spectrum
 of 
\[
L^2(G/H)=  L^2\big(U(p,q)\big/(U(r,s) \times U(p-r, q-s))\big)
\]
 and  of 
\[
L^2(G'/H')=
  L^2\big(U(p-1,q)\big/(U(r,s) \times U(p-r-1, q-s))\big).  
\]

\medskip
In Section~\ref{sec:Arthur}, we conclude the article with some general remarks and illustrate our results with examples of tempered and non-tempered representations, as well as extensions to limits of discrete series representations.

\medskip
Proofs of Theorems~\ref{thm:23081404} and~\ref{thm:24102601} are provided in \cite{HKS},
and details and proofs of the results in Sections~\ref{sec:Spehrep} through~\ref{sec:upq2} will be published in forthcoming articles \cite{xks25, xks25b}.

\medskip
{\bf{Notation:}}\enspace
${\mathbb{N}}=\{0, 1, 2, \dots, \}$, 
${\mathbb{N}_+}=\{1, 2, 3, \dots, \}$,
${\mathbb{R}}_>^n=\{x\in \mathbb{R}^n:x_1>\dots > x_n\}$,
${\mathbb{R}}_\ge^n=\{x\in \mathbb{R}^n:x_1\ge \dots \ge x_n\}$,
${\mathbb{Z}}_\ge^n={\mathbb Z}^n \cap {\mathbb{R}}_\ge^n$.

\vskip 1pc
\par\noindent
{\bf{Acknowledgement.}}
\newline
The authors would like to thank D.~Prasad and A.~Pal for their warm hospitality during the Centennial Conference for Harish-Chandra, 
held at Harish-Chandra Research Institute (HRI) in Prayagraj, India, from October 9 to 14, 2023.

We are also grateful to the Mathematisches Forschungsinstitut Oberwolfach (MFO), 
the Centro Internazionale per la Ricerca Matematica (CIRM) in Trento, 
 the Institut Henri Poincaré (IHP) in Paris, and the Institut des Hautes \'Etudes Scientifiques (IHES) in Bures-sur-Yvette for giving us the opportunity to work together in pleasant conditions.  

The authors thank the anonymous referee for a careful reading of the original manuscript and for their helpful suggestions that improved its readability.

The first author was partially supported by JSPS under Grant-in-Aid for Scientific Research (A) (JP23H00084).

\bigskip
\section{Symmetry Breaking Under Translations}
\label{sec:tfsbo}
Let $G \supset G'$ be any real forms of $GL(n,{\mathbb{C}}) \supset GL(n-1,{\mathbb{C}})$.  

In this section, we discuss the behavior of
\lq\lq{translation functors}\rq\rq\
 for symmetry breaking operators (SBOs)
 between representations of $G$ and $G'$.  

\medskip
\subsection{Harish-Chandra Isomorphism and Translation Functor.}

Let ${\mathfrak{g}}_{\mathbb{C}}={\mathfrak{g l}}(N, {\mathbb{C}})$. 
We shall use $N$ to refer to $n$ or $n-1$ later.  
We set
\begin{equation}
\label{eqn:rhon}
   \rho_{N}:=(\tfrac{N-1}2, \tfrac{N-3}2, \dots, \tfrac{1-N}2).  
\end{equation}
Let ${\mathfrak{Z}}({\mathfrak{g}}_{\mathbb{C}})$ denote the center of the enveloping algebra $U({\mathfrak{g}}_{\mathbb{C}})$.  
We normalize the Harish-Chandra isomorphism
\[
   \operatorname{H o m}_{{\mathbb{C}}\operatorname{-alg}}({\mathfrak{Z}}({\mathfrak{g}}_{\mathbb{C}}), {\mathbb{C}})
   \simeq 
   {\mathbb{C}}^N/{\mathfrak{S}}_N, 
\]
 in such a way that the trivial one-dimensional ${\mathfrak{g}}_{\mathbb{C}}$-module has the infinitesimal character $\rho_N \mod {\mathfrak{S}}_N$.

For a ${\mathfrak{g}}$-module $V$ 
 and for 
$
   \tau \in 
   \invHom{{\mathbb{C}}\operatorname{-alg}}
          {{\mathfrak{Z}}({\mathfrak{g}}_{\mathbb{C}})}
          {{\mathbb{C}}} \simeq {\mathbb{C}}^N/{\mathfrak{S}}_N
$, 
 let $P_{\tau}(V)$ denote the $\tau$-primary component of $V$, 
 that is, 
\[
  P_{\tau}(V)=\bigcup_{k=0}^{\infty}\, \bigcap_{z \in {\mathfrak{Z}}({\mathfrak{g}}_{\mathbb{C}})} \operatorname{Ker}(z - \tau(z))^k.  
\]

Let $\{f_i: i=1, \dots, N \}$ be the standard basis of $\mathbb Z^N$.
We focus on the following translation functors 
 in the Casselman--Wallach category ${\mathcal{M}}(G)$
 or in the category of Harish-Chandra modules:
\begin{equation}
\label{eqn:transl_1}
    \phi_{\tau}^{\tau+\varepsilon f_i}(\cdot) :=
    \begin{cases}
        P_{\tau+f_i}(P_{\tau}(\cdot) \otimes {\mathbb{C}}^N)
   &\text{ if }\varepsilon =+,
   \\
 P_{\tau-f_i}(P_{\tau}(\cdot) \otimes ({\mathbb{C}}^N)^{\vee})
 &\text{ if } \varepsilon =-1.
    \end{cases}
\end{equation}

These functors are particular cases of the translation functors introduced by J.C.~Jantzen and G.~Zuckerman.

\bigskip
\subsection{Non-Vanishing Condition for Translating SBOs.}

Suppose that $\Pi \in {\mathcal{M}}(G)$
 (resp., $\pi \in {\mathcal{M}}(G')$)
 has a ${\mathfrak{Z}}({\mathfrak{g}}_{\mathbb{C}})$-infinitesimal character 
 $\tau \in {\mathbb{C}}^n/{\mathfrak{S}}_n$
 (resp. ${\mathfrak{Z}}({\mathfrak{g}}_{\mathbb{C}}')$-infinitesimal character $\tau' \in {\mathbb{C}}^{n-1}/{\mathfrak{S}}_{n-1}$).

In \cite{HKS}, 
 we have established the following theorems, 
 which provide useful information on symmetry breaking under translations.

\medskip

\begin{theorem}
\label{thm:23081404}
Let $\Pi \in {\mathcal{M}}(G)$ and $\pi \in {\mathcal{M}}(G')$.  
Suppose that any generalized eigenspaces of ${\mathfrak{Z}}({\mathfrak{g}}_{\mathbb{C}})$ in $\Pi \otimes {\mathbb{C}}^{n}$
 are eigenspaces.  
\par\noindent
{\rm{(1)}}\enspace
If  $\invHom{G'}{\Pi|_{G'}}{\pi}\ne \{0\}$, 
 then $\invHom{G'}{\phi_{\tau}^{\tau+f_i}(\Pi)|_{G'}}{\pi}\ne \{0\}$
for any $i$ such that 
$\tau_i \not \in \{\tau_1'-\frac 1 2, \tau_2'-\frac 1 2, \dots, \tau_{n-1}'-\frac 1 2\}$.  
\par\noindent
{\rm{(2)}}\enspace
If $\invHom{G'}{\Pi|_{G'}}{\pi}= \{0\}$, then 
 $\invHom{G'}{\phi_{\tau}^{\tau+f_i}(\Pi)|_{G'}}{\pi}= \{0\}$
 for any $i$ such that 
$
   \tau_i \not \in \{\tau_1'-\frac 1 2, \tau_2'-\frac 1 2, \dots, \tau_{n-1}'-\frac 1 2\}$.  
\end{theorem}

\medskip

\begin{theorem}
\label{thm:24102601}
Let $\Pi \in {\mathcal{M}}(G)$ and $\pi \in {\mathcal{M}}(G')$.  
Suppose that any generalized eigenspaces of ${\mathfrak{Z}}({\mathfrak{g}}_{\mathbb{C}})$ in $\Pi \otimes ({\mathbb{C}}^{n})^{\vee}$ are eigenspaces.  
\par\noindent
{\rm{(1)}}\enspace
If $\invHom{G'}{\Pi|_{G'}}{\pi}\ne \{0\}$, 
 then $\invHom{G'}{\phi_{\tau}^{\tau-f_i}(\Pi)}{\pi}\ne \{0\}$
for any $i$ such that 
 $\tau_i \not \in \{\tau_1'+\frac 1 2, \tau_2'+\frac 1 2, \dots, \tau_{n-1}'+\frac 1 2\}$.  
\par\noindent
{\rm{(2)}}\enspace
If $\invHom{G'}{\Pi|_{G'}}{\pi} = \{0\}$,
 then $\invHom{G'}{\phi_{\tau}^{\tau-f_i}(\Pi)|_{G'}}{\pi}= \{0\}$
 for any $i$ such that 
$
   \tau_i \not \in \{\tau_1'+\frac 1 2, \tau_2'+\frac 1 2, \dots, \tau_{n-1}'+\frac 1 2\}$.  
\end{theorem}

Theorems~\ref{thm:23081404} and~\ref{thm:24102601} are stated under the assumption that \emph{generalized eigenspaces of ${\mathfrak{Z}}({\mathfrak{g}}_{\mathbb{C}})$ in $\Pi \otimes {\mathbb{C}}^{n}$ and $\Pi \otimes ({\mathbb{C}}^{n})^{\vee}$
 are actual eigenspaces}. This condition holds {\it{generically}}; for example, it is always satisfied when $G$ is compact, and it is also satisfied for any Harish-Chandra discrete series representation of $G=U(p,q)$ (see \cite[Prop.~5.5]{HKS}).

Theorems~\ref{thm:23081404} and~\ref{thm:24102601} reveal an intrinsic reason for the appearance of interlacing patterns in certain branching laws, 
such as Weyl's branching law and the Gan--Gross--Prasad conjecture, 
as we discuss in the following section.

Building on
Theorems~\ref{thm:23081404} and~\ref{thm:24102601},
we establish in the remainder of this section two useful results
(Theorems~\ref{thm:24120703} and \ref{thm:25080811}),
based on the concept of \emph{fences} (Definition~\ref{def:int}).

\medskip
\subsection{Interleaving Pattern}
~~~\newline
We set
\begin{align*}
{\mathbb{R}}_{>}^{n}:=&\{x \in {\mathbb{R}}^{n}: x_1 > \cdots > x_n\}, 
\\
{\mathbb{R}}_{\ge}^{n}:=&\{x \in {\mathbb{R}}^{n}: x_1 \ge \cdots \ge x_n\}.
\\
{\mathbb{Z}}_{\ge}^{n}:=&{\mathbb{Z}}^{n} \cap {\mathbb{R}}_{\ge}^{n}.  
\end{align*}

We introduce the notion of \lq\lq{fences}\rq\rq\ as combinatorial objects.
This serves as a refinement of the \lq\lq{walls}\rq\rq\ of the Weyl chambers when we consider the branching for the restriction $G \downarrow G'$,
 where $(G, G')$ are any real forms of
 $(GL(n, \mathbb{C}), GL(n-1, \mathbb{C}))$.  
\begin{dfn}
[Interleaving Pattern and Fence]
\label{def:int}
For $x \in {\mathbb{R}}^n$ and $y \in {\mathbb{R}}^m$, 
 an {\it{interleaving pattern}} $D$
 in ${\mathbb{R}}_{>}^n \times {\mathbb{R}}_{>}^m$
 is a total order 
 among $\{x_1, \dots, x_{n}$, $y_1, \dots, y_{m}\}$, 
 which is compatible 
 with the underlying inequalities
 $x_1 > x_2 > \cdots >x_{n}$ and $y_1 > y_2 > \cdots >y_{m}$.  
For an adjacent inequality between $x_i$ and $y_j$
 such as $x_i > y_j$ or $y_j > x_i$, 
 we refer to the hyperplane in ${\mathbb{R}}^{n+m}$ defined by $x_i = y_j$ as a {\it{fence}}.  
\end{dfn}

By an abuse of notation, 
 we also use the same letter $D$ 
 to denote the region in ${\mathbb{R}}_{>}^n \times {\mathbb{R}}_{>}^m$
 given by its defining inequalities. 
  
Let ${\mathfrak{P}}({\mathbb{R}}^{n, m})$ denote the set of all interleaving patterns in ${\mathbb{R}}_{>}^n \times {\mathbb{R}}_{>}^m$.

\begin{example}
There are 35 interleaving patterns
 for ${\mathbb{R}}_{>}^4 \times {\mathbb{R}}_{>}^3$, 
 such as 
\begin{align*}
 D_1=&\{(x,y)\in {\mathbb{R}}^{4+3}:x_1 > y_1 > x_2  > y_2 > x_3 > y_3 >x_4\}, 
\\
 D_2=&\{(x,y)\in {\mathbb{R}}^{4+3}:y_1 > y_2 > x_1 > x_2 > x_3 > x_4 > y_3\}.  
\end{align*} 
The interleaving pattern $D_1$ is also referred to as the \emph{interlacing pattern}.
There are six fences associated with $D_1$, namely those given by 
\[ x_1=y_1, \ y_1=x_2, \ x_2=y_2, \ y_2=x_3, \ x_3=y_3,
\ \text{and} \ y_3=x_4.
\]
In contrast, there are just two fences associated with $D_2$, namely those given by 
\[
x_1=y_2 \ \text{ and }\ x_4=y_3.
\]
\end{example}

We also consider interleaving patterns
 in ${\mathbb{R}}_{\ge}^{n} \times {\mathbb{R}}_{\ge}^{m}$
 such as $x_1 > y_1 \ge y_2 >x_2$, 
 or those including equalities
 such as 
 $x_1 = y_1 > x_2 = y_2$
 or 
 $x_1 \ge y_1 \ge x_2>y_2$.  
These interleaving patterns will be called {\it{weak interleaving patterns}}.  

\bigskip
\subsection{Stability Theorem for Multiplicities in Symmetry Breaking inside Fences}
~~~\newline
This section establishes a \emph{stability theorem} for the multiplicities in symmetry breaking under coherent continuation.

Let $(G, G')=(GL(n,{\mathbb{R}}), GL(n-1,{\mathbb{R}}))$
 or $(U(p,q), U(p-1, q))$.  

Let ${\mathcal{V}}(G)$ denote the Grothendieck group of ${\mathcal{M}}(G)$, 
 that is, 
 the abelian group generated by $X \in {\mathcal{M}}(G)$ modulo the equivalence relation 
\[
X \sim Y+Z,
\]
 whenever there is a short exact sequence 
 $0 \to Y \to X \to Z\to 0$.

Let $\Pi \colon \xi + {\mathbb{Z}}^n \to {\mathcal{V}}(G)$ be a coherent family of $G$-modules, 
 specifically, 
 $\Pi$ satisfies the following properties:

(1)\enspace
 $\Pi_{\lambda}$ has a ${\mathfrak{Z}}({\mathfrak{g}}_{\mathbb{C}})$-infinitesimal character $\lambda$ if $\lambda \in \xi + {\mathbb{Z}}^n$;

(2)\enspace
$
   \Pi_{\lambda} \otimes F 
   \simeq \underset{\nu \in \Delta(F)}\sum \Pi_{\lambda+\nu}
$
 in ${\mathcal{V}}(G)$
 for any finite-dimensional representation $F$ of $G$.

\begin{theorem}[Stability Theorem in Symmetry Breaking]
\label{thm:24120703}
Suppose that $\Pi \in \operatorname{Irr}(G)$
 has a $\mathfrak{Z}(\mathfrak{g}_\mathbb C)$-infinitesimal character $\xi$
 satisfying
 \begin{equation}
 \label{eqn:reg_xi}
     \xi_i - \xi_{i+1} \ge 1 \quad
 (1 \le i \le n-1).
 \end{equation}
 Let $\Pi \colon \xi + {\mathbb{Z}}^n \to {\mathcal{V}}(G)$
 be the coherent family starting from $\Pi_{\xi}:= \Pi$.
 Let $\nu$ be the infinitesimal character of $\pi \in \operatorname{Irr}(G')$.  
 If $(\xi, \nu)$ satisfies an interleaving pattern 
 $D$ in ${\mathbb{R}}_>^{n} \times {\mathbb{R}}_{\ge}^{m}$, then we have
\[
  [\Pi|_{G'}:\pi] = [\Pi_{\lambda}|_{G'}:\pi]
\]
 for all $\lambda \in \xi +{\mathbb{Z}}^n$
 such that $(\lambda, \nu)$ satisfies the same interleaving pattern $D$.  
\end{theorem}

\begin{remark}
\rm{(1)}\enspace
Such a coherent family exists uniquely because our assumption guarantees that $\xi$ is non-singular.  
\newline
{\rm{(2)}}\enspace
The concept of \lq\lq fences\rq\rq\ is a refinement of the Weyl chambers. 
Hence, if we do not cross the fence,
that is, if $(\lambda, \nu) \in D$,
then $\lambda$ is non-singular 
 and remains in the same Weyl chamber with $\xi$.
Consequently,
 $\Pi_{\lambda}$ is irreducible for any such $\lambda$.
  
\end{remark}

We recall our notation that $\{f_i\}_{1 \le i \le n}$ is the standard basis of ${\mathbb{Z}}^n$.  
To prove Theorem~\ref{thm:24120703}, we introduce the finite set defined by
\[
  {\mathcal{E}}:=\{\pm f_i : 1 \le i \le n\} \subset {\mathbb{Z}}^n.  
\]

\begin{lemma}
\label{lem:24120608}
Let $D \in {\mathfrak{P}}(\mathbb R^{n,m})$.  
For any $(\xi, \nu)$ and $(\lambda, \nu) \in D$
 such that $\lambda-\xi \in {\mathbb{Z}}^n$, 
 there exists a sequence $\lambda^{(j)} \in \xi + {\mathbb{Z}}^n$
 $(j=0,1,2,\dots, N)$
 with the following properties:
\[
  \lambda^{(0)}=\xi, \,\,
  \lambda^{(N)}=\lambda, \,\,
  \lambda^{(j)}-\lambda^{(j-1)} \in {\mathcal{E}}, \,\,
  (\lambda^{(j)}, \nu) \in D
  \quad
  \text{for $1 \le j \le N$.}
\]
\end{lemma}

\begin{proof}
For an interleaving pattern $D$, we define $m(D) \in \{0,1,\dots,n\}$ as follows:
 $m(D):=0$ if $D$ implies $y_1 > x_1$, 
 and otherwise, 
\begin{equation}
\label{eqn:mD}
 m(D):=\text{the largest $i$ such that $x_i>y_1$ in $D$}.  
\end{equation}

There exists a unique element $\mu \in \xi + {\mathbb{Z}}^n$
 such that $(\mu, \nu) \in D$
 and that $\mu$ satisfies the following property
 for any $\lambda \in \xi + {\mathbb{Z}}^n$ with $(\lambda, \nu) \in D$:
\begin{alignat*}{2}
  &\mu_i \le \lambda_i \quad &&\text{if $i \le m(D)$, }
\\
 &\mu_i \ge \lambda_i \quad &&\text{if $m(D)< i \le n$. }
\end{alignat*}

First, 
 we assume that $\lambda=\mu$.  
Then it is readily verified by an inductive argument
 that Lemma \ref{lem:24120608} holds for $\lambda=\mu$.

Second, 
 since the existence of the sequences $\{\lambda^{(j)}\}_{0 \le j \le N}$
 in Lemma \ref{lem:24120608} defines an equivalence relation $\sim$
 among non-singular dominant elements in $\xi+{\mathbb{Z}}^n$, 
 we have $\xi \sim \mu \sim \lambda$, 
 whence the lemma.  
\end{proof}

\begin{proof}
[Proof of Theorem \ref{thm:24120703}]
By Lemma \ref{lem:24120608}, 
 it suffices to prove Theorem \ref{thm:24120703}
 when $\lambda -\xi \in {\mathcal{E}}$.  
For example, 
 suppose that $\lambda-\xi=f_i$ for some $1 \le i \le n$.  
Then, we have
\[
  \xi_i \not \in \{\nu_1-\frac 1 2, \nu_2-\frac 1 2, \dots, \nu_n-\frac 1 2\}
\]
because $(\lambda, \nu)$ and $(\xi, \nu)$ satisfy
 the same interleaving property.

On the other hand, 
 since $\xi$ is non-singular and $\xi_a-\xi_b \in {\mathbb{Z}}$
 for any $1 \le a \le b\le n$, 
 $\xi+f_j$ ($1 \le j \le n$) lies in the same Weyl chamber as $\xi$.  
Therefore, 
 $\Pi_{\xi+f_j} \simeq \phi_{\xi}^{\xi+f_j}(\Pi)$ is either
 irreducible or zero.  
Thus, 
 all the assumptions in Theorem \ref{thm:23081404}
 are satisfied, 
 and we conclude
\[
  [\Pi_{\lambda}|_{G'}:\pi]=[\phi_{\xi}^{\xi+f_j}(\Pi):\pi] \ne 0.  
\]
The multiplicity-freeness theorem concludes 
 that $[\Pi_{\lambda}|_{G'}:\pi]=1$.  
The case $\lambda-\xi=-f_i$ can be proven 
 similarly by using Theorem \ref{thm:24102601}.  
\end{proof}

\bigskip
\subsection{A General Vanishing Theorem for Symmetry Breaking}
~~~\newline
The stability theorem for multiplicities in symmetry breaking (Theorem~\ref{thm:24120703}) leads to a general vanishing theorem for symmetry breaking, for which we provide a proof in this section (Theorem~\ref{thm:25080811}).
 The theorem is formulated in terms of $\tau$-invariants of representations, 
 which we briefly recall below.  

For a non-singular weight $\xi$, the set of integral roots with respect to $\xi$ is defined by
\begin{align*}
    R(\xi):=\,&\{\alpha \in \Delta({\mathfrak{g}}_{\mathbb{C}}, {\mathfrak{j}}_{\mathbb{C}})
              :\langle \alpha^{\vee}, \xi \rangle \in {\mathbb{Z}}\}
\\
         =\,&\{e_i-e_j: i \ne j, \ \xi_i-\xi_j \in {\mathbb{Z}}\}.  
\end{align*}

The weight $\xi$ is \emph{integral} if  
$
   R(\xi)=\Delta({\mathfrak{g}}_{\mathbb{C}}, {\mathfrak{j}}_{\mathbb{C}}).
$
We define the system of positive integral roots
 with respect to $\xi$ by 
\begin{equation*}
    R^+(\xi):=\{e_i-e_j:\xi_i-\xi_j \in {\mathbb{N}}_+\}, 
\end{equation*}
and denote by $\Psi^+(\xi)$ the corresponding set of simple roots in $R^+(\xi)$.

\begin{dfn}[$\tau$-invariant]
\label{def:tau_inv} 
Suppose that $\Pi \in \operatorname{Irr}(G)$ has a non-singular 
 ${\mathfrak{Z}}({\mathfrak{g}}_{\mathbb{C}})$-infinitesimal character $\xi$.
 A simple root $\alpha \in \Psi(\xi)$ is called a $\tau$-{\emph{invariant}}
 of the representation $\Pi$ if 
 \[\Pi_{\xi+ \mu}=0
 \ \text{ for every } \
 \mu \in {\mathbb{Z}}^n
\ \text{ such that } \
  \langle \alpha^\vee, \xi + \mu \rangle =0.
\]
Here
  $\Pi \colon \xi + {\mathbb{Z}}^n \to {\mathcal{V}}(G)$
 denotes the coherent continuation so that 
$\Pi_{\xi}$ is the originally given representation $\Pi$.
\end{dfn}

We denote by $\tau(\Pi) \subset \Psi^+(\xi)$
 the set of $\tau$-invariants of $\Pi$.

In what follows, 
 we assume for simplicity
 that $\xi$ is dominant integral;
 that is, 
$R(\xi) = \Delta({\mathfrak{g}}_{\mathbb{C}}, {\mathfrak{j}}_{\mathbb{C}})$
 and $\Psi^+(\xi)=\{e_1-e_2, \dots, e_{n-1}-e_n\}$.

\begin{theorem}
[Vanishing Theorem]
\label{thm:25080811}
Suppose that $\Pi \in \operatorname{Irr}(G)$
 has a non-singular ${\mathfrak{Z}}({\mathfrak{g}}_{\mathbb{C}})$-infinitesimal character $\lambda \in {\mathbb{R}}_>^n$.  
Let $D$ be an interleaving pattern in ${\mathbb{R}}_>^n \times {\mathbb{R}}_{\ge}^{n-1}$.

Assume that there exists $i$ $(1 \le i \le n-1)$ 
 such that the following two conditions hold:
\begin{enumerate}
\item[(1)] $e_i-e_{i+1} \in \tau(\Pi)$;
\item[(2)]
 $\lambda_i \lambda_{i+1}$ appears 
 as an adjacent string;
 in other words, 
 there is no $\nu_j$
 such that the inequality $\lambda_i > \nu_j > \lambda_{i+1}$
 is allowed in $D$.  
\end{enumerate}

Then 
\[
   [\Pi|_{G'}:\pi]=0.  
\]
\end{theorem}

\begin{proof}
[Proof of Theorem \ref{thm:25080811}]
Let $\Pi \colon \lambda + {\mathbb{Z}}^n \to {\mathcal{V}}(G)$ be 
 the coherent family starting from $\Pi_{\lambda} = \Pi$.  

Suppose that $\lambda_i \lambda_{i+1}$ appears
 as an adjacent string 
 in the interleaving pattern $D$ for $(\lambda, \nu)$.  
Then, 
 there exists $\mu \in {\mathbb{Z}}^n$
 such that the following three conditions are satisfied:
\begin{itemize}
\item[$\bullet$]
$\langle \lambda+\mu, e_i-e_{i+1} \rangle =0$;
\item[$\bullet$]
$\langle \lambda+\mu, \beta \rangle \ne 0$
 for any $\beta \in \Delta^+({\mathfrak{g}}_{\mathbb{C}}, {\mathfrak{j}}_{\mathbb{C}})\setminus \{e_i-e_{i+1}\}$;
\item[$\bullet$]
the pair $(\lambda+\mu, \nu)$ satisfies the same interleaving pattern $D$.  
\end{itemize}

As in Lemma~\ref{lem:24120608}, 
 there exists a sequence $\lambda^{(j)} \in \lambda+ {\mathbb{Z}}^n$
 $(j=0, 1, \dots, N)$
 with the following properties:
\begin{itemize}
\item[]
$\lambda^{(0)}$, $\dots$, $\lambda^{(N-1)}$ are non-singular;
\item[]
$\lambda^{(0)}=\lambda$, $\lambda^{(N)}=\lambda+\mu$;
\item[]
$\lambda^{(j)}-\lambda^{(j-1)} \in {\mathcal{E}}$, 
$(\lambda^{(j)}, \nu) \in D$\quad for all $1 \le j \le N$.
\end{itemize}
We note that $\lambda^{(0)}, \dots, \lambda^{(N-1)}$ satisfy the regularity assumption~\eqref{eqn:reg_xi}
in Theorem~\ref{thm:24120703},
whereas $\lambda^{(N)}$ does not.
It then follows from the proof of Theorem~\ref{thm:24120703} that we obtain
\begin{equation}
\label{eqn:vanish_tau}
   [\Pi_{\lambda^{(0)}}|_{G'}:\pi]
=\cdots=[\Pi_{\lambda^{(N-1)}}|_{G'}:\pi]
                  \le [\Pi_{\lambda^{(N)}}|_{G'}:\pi].  
\end{equation}
On the other hand, 
 since $e_i-e_{i+1} \in \tau(\Pi)$, 
 $\Pi_{\lambda^{(N)}}=\Pi_{\lambda+\mu}=0$.  
Thus, 
 the right-hand side of \eqref{eqn:vanish_tau} is zero, 
 and hence, 
 the theorem is proved.  
\end{proof}

Theorem~\ref{thm:25080811}, when applied to the pair of compact Lie groups $(G, G')=(U(n), U(n-1))$, yields a new proof of the necessity of the interlacing pattern \eqref{eqn:Weylhw} below in Weyl's branching laws, as we shall see in the next section.

\bigskip
\section{Known Examples for $(G, G')=(U(p,q), U(p-1,q))$}
\label{sec:upq}
We begin in this section by demonstrating how Theorems \ref{thm:23081404} and \ref{thm:24102601} provide a new perspective in  the interlacing patterns that appear in known examples of branching laws,
 such as Weyl's branching law for finite-dimensional representations regarding the restriction $U(n) \downarrow U(n-1)$
 and the patterns \cite{He} in the Gan--Gross--Prasad conjecture regarding the branching of discrete series representations for the restriction $U(p,q)\downarrow U(p-1,q)$.

The cases in the branching of non-tempered representations for the restriction $GL(n,{\mathbb{R}}) \downarrow GL(n-1,{\mathbb{R}})$
 are more involved, 
 which we will discuss in Section \ref{sec:Spehrep}
 through Section \ref{sec:GHgeneral}, 
 along with the phenomenon of \emph{jumping fences}.  
 We revisit the branching for $U(p,q) \downarrow U(p-1,q)$ by considering
non-tempered representations in Section \ref{sec:upq2}.  

\bigskip
\subsection{Weyl's Branching Law for $U(n) \downarrow U(n-1)$}
\label{subsec:Weylbranch}
~~~\newline
We begin by illustrating the concept of \emph{fences} (Definition~\ref{def:int}) with the branching for \emph{finite-dimensional} representations.

Let $\Kirredrep{G}{x}$ denote the irreducible finite-dimensional representation of $G := U(n)$
 with highest weight $x \in {\mathbb{Z}}_{\ge}^{n}$ in the standard coordinates.  
Similarly, 
let $\Kirredrep {G'}{y}$ denote the irreducible representation of $G'=U(n-1)$ with highest weight $y \in {\mathbb{Z}}_{\ge}^{n-1}$.

According to Weyl’s branching law,
the restriction of an irreducible representation $\Kirredrep{G}{x}$ of $G$ to the subgroup $G'$ contains the irreducible representation $\Kirredrep{G'}{y}$ of $G'$, that is,
\begin{equation}
\label{eqn:Weyl_law}
[\Kirredrep{G}{x}|_{G'}:\Kirredrep{G'}{y}] \ne 0, 
\
\text{ (equivalently,} 
\
[\Kirredrep{G}{x}|_{G'}:\Kirredrep{G'}{y}]=1)
\end{equation}
if and only if the highest weights satisfy the interlacing inequalities:
\begin{equation}
\label{eqn:Weylhw}
   x_1 \ge y_1 \ge x_2 \ge y_2 \ge \cdots \ge x_{n-1} \ge y_{n-1} \ge x_n.  
\end{equation}

This section reinterprets the above classical result from the new perspective of 
\emph{translation for symmetry breaking}, as formulated in Theorems~\ref{thm:23081404} and \ref{thm:24102601}.
To this end, 
 we reformulate the condition (\ref{eqn:Weylhw})
 in terms of the infinitesimal characters. 
 
Recall from \eqref{eqn:rhon} that
\[\rho_n =\frac 1 2 (n-1, \dots, 1-n),
 \quad\text{ and }\quad 
 \rho_{n-1} =\frac 1 2 (n-2, \dots, 2-n).
 \]
Then the ${\mathfrak{Z}}({\mathfrak g}_{\mathbb C})$-infinitesimal character of the $G$-module $\Kirredrep{G}{x}$ and the ${\mathfrak{Z}}({\mathfrak g}'_{\mathbb C})$-infinitesimal character of of the $G'$-module $\Kirredrep{G'}{y}$ are given, respectively, by $\tau \bmod {\mathfrak S}_n$ and $\tau' \bmod {\mathfrak S}_{n-1}$, where
 \[
 \tau:= x + \rho_n
 \quad \text{ and } \quad
  \tau':=y + \rho_{n-1}.
 \]
Thus, the inequality \eqref{eqn:Weylhw} for highest weights is equivalent to 
 the following strict inequality:
\begin{equation}
\label{eqn:Weylint}
   \tau_{1}>\tau_{1}'> \tau_2>\tau_2'> \cdots >\tau_{n-1}>\tau_{n-1}'>\tau_{n}.  
\end{equation}

\medskip
We begin with an observation that translation functors within the same dominant chamber can easily alter the multiplicity $[\Kirredrep{G}{x}|_{G'}:\Kirredrep{G'}{y}]$ in symmetry breaking.
\begin{example}
\label{ex:alter_mult}
Consider $(G,G')=(U(3), U(2))$.
Let 
\[F_1:=\Kirredrep{G}{1,0,0}, 
\quad F_2:=\Kirredrep{G}{1,1,0}, 
\quad  F':=\Kirredrep{G'}{0,0},
\]
that is, $F_1$, $F_2$, and $F'$ are, respectively,
the standard representation $\mathbb C^3$ of $G$, its exterior representation $\Lambda^2 \mathbb C^3$, and the trivial representation of $G'=U(2)$.

The $\mathfrak{Z}(\mathfrak{g}_\mathbb C)$-infinitesimal character of $F_1$ and $F_2$ are given by
\[\tau_1=(2,0,-1) \quad \text{and} \quad \tau_2=(2,1,-1)
\bmod \mathfrak{S}_3.
\]
respectively.
Then the translation functors, as defined in \eqref{eqn:transl_1},
are performed without crossing walls,
that is, both $\tau_1=\tau_2-f_2$ and $\tau_2=\tau_1+f_2$ lie in the same (strict) Weyl chamber, giving
\[
    \phi_{\tau_2}^{\tau_2-f_2}(F_2) \simeq F_1,
    \qquad
 \phi_{\tau_1}^{\tau_1+f_2}(F_1) \simeq F_2.
\]
However, the multiplicities are given as
\[
  [F_1|_{G'}:F']=1, \qquad
  [F_2|_{G'}:F']=0. 
\]
This shows that translation functors can significantly alter the nature of symmetry breaking even inside the Weyl chamber.
\end{example}

On the other hand, Theorems \ref{thm:23081404} and \ref{thm:24102601} are formulated in terms of \emph{fences}, rather than in terms of the usual notion of \emph{walls} for Weyl chambers.

We now explain how these theorems recover the sufficiency of the interlacing property \eqref{eqn:Weylhw} for highest weights (or, equivalently, \eqref{eqn:Weylint} for infinitesimal characters) in a simple and specific case.

To see this, suppose that we are given any $y \in {\mathbb{Z}}_{\ge}^{n-1}$ and any $x_n$ such that $y_{n-1} \ge x_n$.  
We set 
\[
\widetilde x:=(y_1, \dots, y_{n-1}, x_n)\in {\mathbb{Z}}_{\ge}^n.
\]
Then 
 $[\Kirredrep{G}{\widetilde x}|_{G'}:\Kirredrep{G'}{y}]\ne 0$
 because the highest weight vector of $\Kirredrep{G}{\widetilde x}$ generates the irreducible $G'$-submodule $\Kirredrep{G'}{y}$.  

We now apply Theorem
\ref{thm:24102601} to $\pi:=\Kirredrep{G'}{y} \in {\mathcal{M}}(G')$, and consider the translation functors
 for ${\mathcal{M}}(G)$.  
Due to the integral condition
 $\tau_i-\tau_j \in {\mathbb{Z}}$
 for all $1 \le i, j \le n$, 
 the translation 
 \[\tau \rightsquigarrow \tau+ \varepsilon f_i
 \quad
 (\varepsilon =+1 \ \textrm{ or } -1)
 \]
 does not cross any wall of the same Weyl chamber of $G$;
 hence, the translation $\phi_{\tau}^{\tau+\varepsilon f_i}(\Kirredrep{G}{x})$
 is either 0 or irreducible.  
More precisely, 
 \[
 \phi_{\tau}^{\tau+\varepsilon f_i}(\Kirredrep {G}{x}) \simeq \Kirredrep{G}{x +\varepsilon f_i}
 \]
 if $x_i \ne x_{i-\varepsilon}$.  
Therefore, 
 an iterated application of Theorem \ref{thm:24102601} implies
 that 
\[
   [\Kirredrep{G}{x}|_{G'}: \Kirredrep{G'}{y}]\ne 0
\]
as long as the pair $(x + \rho_n, y + \rho_{n-1})$ satisfies \eqref{eqn:Weylint}.

\medskip

We now give a new proof for the necessity of 
the interlacing property \eqref{eqn:Weylhw},
using the vanishing results based on
Theorem~\ref{thm:25080811}.

Let $\Pi$ be an irreducible finite-dimensional representation of $G=U(n)$ with highest weight $x$, 
 and let $\pi$ be an irreducible representation of $G'=U(n-1)$ with highest weight $y$.

Suppose that $(x,y)$ does not satisfy the interlacing property~\eqref{eqn:Weylhw}, 
 or equivalently, 
 that their infinitesimal characters $\lambda=x+ \rho_n$
 and $\nu=y+\rho_{n-1}$ do not satisfy
 the interlacing property~\eqref{eqn:Weylint}.  
This implies 
 that there exists $i$ $(1 \le i \le n-1)$
 such that $\lambda_i \lambda_{i+1}$ appears
 as an adjacent string
 in their interleaving pattern for $(\lambda, \nu)$.  
However, 
 for a finite-dimensional representation $\Pi$, 
 all simple roots constitute its $\tau$-invariant.  
Hence, 
 it follows from Theorem~\ref{thm:25080811}
 that 
\[
   [\Pi|_{G'}: \pi]=0.  
\]
This completes the proof of the reverse implication.

\bigskip
\subsection{Gan--Gross--Prasad conjecture
 for $U(p,q)\downarrow U(p-1,q)$}
~~~\newline
In the non-compact setting $(G, G')=(U(p,q), U(p-1, q))$, 
 an analogous interleaving property to \eqref{eqn:Weylint} arises, 
 which we now recall.  

Let $G=U(p,q)$ and $K=U(p) \times U(q)$.  
The complexifications are given
 by $G_{\mathbb{C}}=GL(p+q, {\mathbb{C}})$
 and $K_{\mathbb{C}}=GL(p, {\mathbb{C}}) \times GL(q, {\mathbb{C}})$, 
 respectively.  
Let $W_G ={\mathfrak{S}}_{p+q}$
 and $W_K ={\mathfrak{S}}_{p} \times {\mathfrak{S}}_{q}$
 be the Weyl groups for the root systems $\Delta({\mathfrak{g}}_{\mathbb{C}})$
 and $\Delta({\mathfrak{k}}_{\mathbb{C}})$, 
 respectively.  
We define
\[
  W^{\mathfrak{k}}:=
\{
   w \in W_G: \text{$w \nu$ is $\Delta^+({\mathfrak{k}})$-dominant 
for any $\Delta({\mathfrak{g}})$-dominant $\nu$}
\}.  
\]
This means that 
$w \in W^{\mathfrak{k}}$
 if $w \in W_G={\mathfrak{S}}_{p+q}$ satisfies 
$
  w^{-1}(i) < w^{-1}(j)
$
whenever $1\le i < j \le p$ or $p+1 \le i < j \le p+q$.

Then $W^{\mathfrak{k}}$ is the set of complete representatives of 
 $W_K \backslash W_G$, 
 which parametrizes closed $K_{\mathbb{C}}$-orbits
 on the full flag variety of $G_{\mathbb{C}}$.  
We further define
\[
 C_+:=\,\{x \in {\mathbb{R}}^{p+q}: x_1> \cdots >x_{p+q}\}.  
\]

For $w \in W$, 
the set 
$
  w C_+
$ defines an interlacing pattern in ${\mathbb{R}}_{>}^p \times {\mathbb{R}}_{>}^q$:
\[
w C_+=
\{x \in {\mathbb{R}}^{p+q}: x_{i_1} > x_{i_{2}} > \cdots > x_{i_{p+q}}\}.
\]

For $\varepsilon \in \frac 1 2 {\mathbb{Z}}$, 
 we define
\begin{align}
\notag
  {\mathbb{Z}}_{\varepsilon}
  &:=
  {\mathbb{Z}}+\varepsilon.  
\\
\notag
  ({\mathbb{Z}}_{\varepsilon})_{\operatorname{reg}}^{p+q}
  &:=\{x \in ({\mathbb{Z}}_{\varepsilon})^{p+q}:x_i \ne x_j \,\,\text{ if $i \ne j$}\}, 
\\
\notag
 ({\mathbb{Z}}_{\varepsilon})_{>}^{p+q}
  &:=\{x \in ({\mathbb{Z}}_{\varepsilon})^{p+q}: x_1 > \cdots > x_{p+q}\}, 
\\
\label{eqn:Zint}
({\mathbb{Z}}_{\varepsilon})_{>}^{p, q}
  &:=\{x \in ({\mathbb{Z}}_{\varepsilon})_{\operatorname{reg}}^{p+q}:
x_1>\cdots>x_p \text{ and } x_{p+1} > \cdots > x_{p+q}\}.  
\end{align}

Let $\operatorname{Disc}(G)$ denote
 the set of discrete series representations of $G$, 
 which is parametrized for $G=U(p,q)$ as follows:
 let $\varepsilon:=\frac 1 2(p+q-1)$.  
\begin{equation*}
   \operatorname{Disc}(G)
   \simeq
   ({\mathbb{Z}}_{\varepsilon})_{>}^{p, q}
  \simeq
  ({\mathbb{Z}}_{\varepsilon})_{>}^{p+q} \times W^{\mathfrak{k}}, 
\quad
  \Pi_{\lambda}=\Pi^w(\lambda^+)
  \leftrightarrow
  \lambda
  \leftrightarrow
   (\lambda^+, w), 
\end{equation*}
where $\lambda = w \lambda^+$.  
The geometric meaning of $w$ is that the support of the localization of the $({\mathfrak{g}}, K)$-module $\Pi^w(\lambda^+)_K$
 via the Beilinson--Bernstein correspondence using ${\mathcal{D}}$-modules
 is the closed $K_{\mathbb{C}}$-orbit that corresponds to $w$, 
 while $\lambda$ is the Harish-Chandra parameter, 
 in particular,
 \[
 \lambda \equiv \lambda^+ \mod {\mathfrak{S}}_n
 \] is its ${\mathfrak{Z}}({\mathfrak{g}}_{\mathbb{C}})$-infinitesimal character.

Let $G=U(p,q)$ and $G'=U(p-1,q)$.  
We set
\[
   \text{$\varepsilon=\frac 1 2(p+q-1)$ \quad and \quad $\varepsilon'=\frac 1 2(p+q-2)$.}  
\]
The classification of a pair $(\Pi, \pi) \in \operatorname{Disc}(G) \times \operatorname{Disc}(G')$
 such that $[\Pi^{\infty}|_{G'}:\pi^{\infty}] \ne 0$ can be described by the parameters  
\[
   (\lambda^+, \nu^+) \in ({\mathbb{Z}}_{\varepsilon}^{p+q})_{>} \times ({\mathbb{Z}}_{\varepsilon'}^{p+q-1})_{>}
\]
 such that $[\Pi^w(\lambda^+)|_{G'}:\pi^{w'}(\nu^{+})] \ne 0$
 for each $(w, w') \in W^{\mathfrak{k}}$.

He \cite{He} determined all such pairs $(\lambda^+, \nu^+)$, 
 relying on the combinatorics 
 of the theta correspondence.  
In his theorem, 
 certain interleaving patterns
 of $(\lambda^+, \nu^+)$ appears.  
The following theorem explains an intrinsic reason for these interleaving patterns, 
from a different perspective, 
 using \lq\lq{translation functor for symmetry breaking}\rq\rq,  
 and reveals why interleaving patterns occur in the context of the Gan--Gross--Prasad conjecture.

For an interleaving pattern $D \in {\mathfrak{P}(\mathbb{R}^{p+q,p+q-1})}$
 (Definition \ref{def:int}), 
 we set
\[
  D_{\operatorname{int}}:=
  D \cap ({\mathbb{Z}}_{\varepsilon}^{p+q} \times {\mathbb{Z}}_{\varepsilon'}^{p+q-1}).  
\]

\begin{theorem}
[$U(p, q) \downarrow U(p-1, q)$]
\label{thm:24102706}
Fix $w \in W^{{\mathfrak{k}}}$, $w' \in W^{{\mathfrak{k}}'}$
 and an interleaving pattern  $D \in {\mathfrak{P}(\mathbb{R}^{p+q,p+q-1})}$.  
Then the following two conditions
 on the triple $(w, w', D)$
 are equivalent:
\newline
{\rm{(i)}}\enspace
$[\Pi^w(\lambda^+)|_{G'}:\pi^{w'}(\nu^+)] \ne 0$
 for some $(\lambda^+, \nu^+) \in D_{\operatorname{int}}$, 
\newline
{\rm{(ii)}}\enspace
$[\Pi^w(\lambda^+)|_{G'}:\pi^{w'}(\nu^+)] \ne 0$
 for all $(\lambda^+, \nu^+) \in D_{\operatorname{int}}$.  
\end{theorem}

Theorem \ref{thm:24102706} is derived from the iterated application
 of Theorems~\ref{thm:23081404} and~\ref{thm:24102601}, 
 along with the use of a spectral sequence
 for cohomological parabolic induction.

\begin{example}[Holomorphic Discrete Series, Thm.~8.11 in \cite{xkProg2007}]
For $\nu \in ({\mathbb{Z}}+\frac{p+1-1}{2})^{p+q}$
subject to the condition
\[\nu_p > \cdots> \nu_{p+q-1} > \nu_1 > \cdots >\nu_{p-1},
\]
 let $\pi(\nu)$ denote the corresponding holomorphic discrete series representation of $G'$.  
Take $\widetilde \lambda \in (\mathbb{Z}+\frac{p+q+1}{2})^{p+q}$
 such that 
\[
 \widetilde \lambda_{p+j}:=\nu_{p+j-1}-\frac 1 2\quad
(1 \le j \le q),
\]
\[
\widetilde {\lambda_1} > \nu_1 > \cdots > \nu _{p-1} > \widetilde {\lambda_p}.  
\]
For this pair $(\tilde{\lambda}, \nu)$, one readily sees that $\pi(\nu)$ occurs 
in $\Pi(\widetilde \lambda)|_{G'}$
 as the \lq\lq{bottom layer}\rq\rq.  
 
For a general pair $(\lambda, \nu)$
 an iterated application of Theorems \ref{thm:23081404} and \ref{thm:24102601}
 implies
 that $[\Pi(\lambda)|_{G'}:\pi(\nu)]\ne 0$ whenever $\lambda \in {\mathbb{Z}}_{\varepsilon}^{p+q}$ satisfies the same interlacing condition
\[
   \nu_p>\lambda_{p+1}>\cdots>\nu_{p+q-1}>\lambda_{p+q}>
   \lambda_1>\nu_1>\cdots>\nu_{p-1}>\lambda_p.  
\]
\end{example}

\bigskip
\section{Branching of Some Special Representations 
 for $GL(2m,{\mathbb{R}}) \downarrow GL(2m-1,{\mathbb{R}})$}
\label{sec:Spehrep}
In this section, 
 we explore an application 
 of Theorem \ref{thm:24120703}
 to a family
 of non-tempered representations
 of $G=GL(2m, {\mathbb{R}})$, 
 see \cite{Sp}, 
 when restricted to the subgroup $G'=GL(2m-1, {\mathbb{R}})$. A detailed proof can be found in  \cite{xks25b}. 

A key aspect
 in applying these theorems is a parity condition 
 that $\lambda_i-\nu_j \in {\mathbb{Z}}+\frac 1 2$
 for every $1 \le i \le p+q$
 and $1 \le j \le p+q-1$.  

We shall see  in Section~\ref{subsec:jump}
 that a phenomenon of \lq\lq{jumping the fences of interlacing patterns}\rq\rq\ naturally arises for the parity conditions on $\lambda$ and $\nu$
 such that $\lambda_i - \nu_j \not \in {\mathbb{Z}}+ \frac 1 2$.

\medskip
\subsection{A Family of Representations of $GL(2m, {\mathbb{R}})$}
For $\varepsilon \in \{0,1\}$, 
 let 
\[
   \Pi \colon ({\mathbb{Z}}+\varepsilon)^{2m} \to {\mathcal{V}}(G)
\]
be the coherent family of smooth representations 
 such that $\Pi(\lambda)$ is the smooth representation
 of a special unitary representation
 studied in \cite{Sp}, 
 sometimes referred to as the {\it{$\ell$-th Speh representation}}, 
 if
\[
  \lambda= \frac 1 2 (\ell, \dots, \ell, -\ell, \dots, -\ell)
           +(\rho_m, \rho_m)
\quad
\text{for $1 \le \ell$}.  
\]
Here, we recall from \eqref{eqn:rhon} 
$
   \rho_m=(\frac{m-1}2, \dots, \frac{1-m}2)
$.
The parity $\varepsilon$ and $\ell$ is related by $\ell+2 \varepsilon +m+1 \in 2{\mathbb{Z}}$.

There is a $\theta$-stable parabolic subalgebra
 ${\mathfrak{q}}={\mathfrak{l}}_{\mathbb{C}}+{\mathfrak{u}}$
 of ${\mathfrak{g}}_{\mathbb{C}} = {\mathfrak{gl}}(2m,{\mathbb{C}})$, 
 unique up to an inner automorphism 
 of $G=GL(2m, {\mathbb{R}})$, 
 such that the real Levi subgroup $N_G({\mathfrak{q}})$
 is isomorphic to $L:=GL(m,{\mathbb{C}})$.  
The underlying $({\mathfrak{g}}, K)$-module of $\Pi(\lambda)$
 is obtained by a cohomological parabolic induction from an irreducible finite-dimensional representation $F_{\lambda}$ of ${\mathfrak{q}}$, 
 on which the unipotent radical ${\mathfrak{u}}$ acts trivially and $L$ acts by 
\[
  \Kirredrep{GL(m,{\mathbb{C}})}{\lambda'-\rho_m} 
  \otimes 
  \overline{\Kirredrep{GL(m,{\mathbb{C}})}{\lambda''-\rho_m}}.
\]
Here $\lambda=(\lambda', \lambda'') \in ({\mathbb{Z}}+\varepsilon)^m \times ({\mathbb{Z}}+\varepsilon)^m$.

The representation $\Pi({\lambda})$ of $G$ is irreducible
 if 
\[
   \lambda_1 > \lambda_2 > \cdots > \lambda_{2m}, 
\]
and is unitarizable
 if 
$
   \lambda_1= \cdots =\lambda_m
 =-\lambda_{m+1}=\cdots=-\lambda_{2m}.
$

\medskip
\subsection{A Family of Representations of $GL(2m-1,{\mathbb{R}})$}
~~~\newline
Let $L' := GL(1,{\mathbb{R}}) \times GL(m-1,{\mathbb{C}})$
be a subgroup of $G' := GL(m-1,{\mathbb{R}})$.  
For 
\[
   \nu \equiv (\nu', \nu_m, \nu'') \in ({\mathbb{Z}}+\varepsilon + \tfrac{1}{2})_{>}^{m-1} 
   \times {\mathbb{C}} 
   \times ({\mathbb{Z}}+\varepsilon + \tfrac{1}{2})_{>}^{m-1},
\]
and $\kappa \in \{0,1\}$, 
let $F_{\kappa}'(\nu)$ denote the irreducible finite-dimensional $L'$-module
given by 
\[
   F_{\kappa}'(\nu) = \chi_{\nu_m, \kappa} \boxtimes W_{\nu', \nu''},
\]
where
\begin{align*}
  \chi_{\nu_m, \kappa}(x)
    &:= |x|^{\nu_m} (\operatorname{sgn} x)^{\kappa},
    \quad \text{for } x \in GL(1, {\mathbb{R}}) \simeq {\mathbb{R}}^{\times}, \\[4pt]
  W_{\nu', \nu''}
    &:= \Kirredrep{GL(m-1,{\mathbb{C}})}{\nu'-\rho_{m-1}}
        \otimes
        \overline{\Kirredrep{GL(m-1,{\mathbb{C}})}{\nu''-\rho_{m-1}}}.
\end{align*}

There exists a $\theta$-stable parabolic subalgebra ${\mathfrak{q}}'={\mathfrak{l}}_{\mathbb{C}}'+{\mathfrak{u}}'$
 of  ${\mathfrak{g}}_{\mathbb{C}}'={\mathfrak{gl}}(2m-1, {\mathbb{C}})$, 
 unique up to an inner automorphism 
 of $G'=GL(2m-1, {\mathbb{R}})$, 
 such that the real Levi subgroup $N_{G'}({\mathfrak{q}}')$ is isomorphic to $L'$.  
Let $\pi_{\kappa}(\nu)$ denote the smooth admissible representation of $G'$
 whose underlying $({\mathfrak{g}}', K')$-module is
 isomorphic to the cohomological parabolic induction from the irreducible 
 finite-dimensional representation $F_{\kappa}'(\nu)$.  
In our normalization, 
 $\nu \mod \mathfrak{S}_{2m-1}$ coincides
 the ${\mathfrak{Z}}({\mathfrak{g}}_{\mathbb{C}}')$-infinitesimal character 
 of $\pi_{\kappa}(\nu)$.  
The $({\mathfrak{g}}', K')$-module $\pi_{\kappa}(\nu)$ is unitarizable
 if $$\nu_m \in \sqrt{-1}{\mathbb{R}}, \ 
 \nu'=c{\bf{1}}_{m-1}+\rho_{m-1}, \
 \text{ and } \nu''=-c{\bf{1}}_{m-1}+\rho_{m-1}$$
 for some $c \in \frac 1 2 {\mathbb{N}}$.

We write simply $\pi(\nu)$ for $\pi_{\kappa}(\nu)$
 when $\nu_m \in {\mathbb{Z}}$
 and when  
\begin{equation}
\label{eqn:24122003}
   \kappa + \nu_m + 2 \varepsilon +m-1 \in 2{\mathbb{Z}}.  
\end{equation}

\medskip
\subsection{Branching for $GL(2m,{\mathbb{R}})\downarrow GL(2m-1,{\mathbb{R}})$}
~~~\newline
In the same spirit as the reinterpretation 
 of Weyl's classical branching laws from the perspective
 of \lq\lq{translation for symmetry breaking}\rq\rq, 
 as explained 
 in Section \ref{subsec:Weylbranch}, 
 we derive the following theorem 
 starting from a \lq\lq{simpler case}\rq\rq, 
 that is, 
 when 
\begin{align}
\label{eqn:24122004}
  &\lambda_{1} > \nu_1, \quad \nu_{2m-1} > \lambda_{2m}, \quad
   \lambda_1+\lambda_{2m} = \nu_m,
\\
\notag  
  &\lambda_{i+1}= \nu_i-\frac 1 2\,\,
  (1 \le i \le m-1), \,\,\,
  \lambda_i=\nu_i+\frac 1 2
  \,\,
  (m+1 \le i \le 2m-1).  
\end{align}

We note
 that such $(\lambda,\nu)$ lies in the interleaving pattern:
\begin{equation}
\label{eqn:lmdnuilp}
   \lambda_1 > \nu_1 > \lambda_2 > \cdots > \nu_{m-1} >\lambda_m
  > 
  \lambda_{m+1}>\nu_{m+1} > \cdots > \nu_{2m-1} >\lambda_{2m}.  
\end{equation}

\begin{theorem}
[{\cite{xks25b}}]
\label{thm:24120248}
Let $(G,G')=(GL(2m, {\mathbb{R}}), GL(2m-1, {\mathbb{R}}))$.  
Let $\varepsilon \in \{0,\frac 1 2\}$, 
$\lambda \in ({\mathbb{Z}}+\varepsilon)_{>}^{2m}$, 
 and $\nu \in ({\mathbb{Z}}+\varepsilon+\frac 1 2)_{>}^{m-1} 
         \times {\mathbb{Z}} \times 
        ({\mathbb{Z}}+\varepsilon+\frac 1 2)_>^{m-1}$
 satisfying
$
  \nu_{m-1}> \nu_m > \nu_{m+1}$
 and 
$
  \nu_{m-1}- \nu_{m+1} \ne 1$.  

If $(\lambda,\nu)$ satisfies \eqref{eqn:lmdnuilp}, 
 then 
\[
   [\Pi({\lambda})|_{G'}:\pi(\nu)]=1.  
\]
\end{theorem}

\bigskip

\section{A Non-Vanishing Theorem for Period Integrals}
\label{sec:period}
~~~
In general, proving the non-vanishing of symmetry breaking is a difficult problem. 
However, thanks to the general results in Theorems~\ref{thm:23081404} and~\ref{thm:24102601} 
(see also Theorem~\ref{thm:24120703}), 
it suffices to consider representations for only specific parameter pairs $(\lambda, \nu)$ when $(\mathfrak{g}_{\mathbb C}, \mathfrak g'_{\mathbb C})=(\mathfrak{gl}(n, \mathbb C), \mathfrak{gl}(n-1, \mathbb C))$.  

We therefore focus on developing a method for detecting the existence of a non-zero symmetry breaking operator for such representations.
To this end, we consider the situation where a pair of reductive groups $G' \subset G$ induces a natural embedding of their symmetric spaces $G'/H' \subset G/H$ and restrict $H$-distinguished representations $\Pi \in \operatorname{Irr}(G)$ to the subgroup $G'$.
The classification of triples $G' \subset G \supset H$ for which the restriction $\Pi|_{G'}$ has the \emph{uniformly bounded multiplicity property} has been recently accomplished in \cite{xk22}.

In this section, we propose a method for detecting the existence of a non-zero symmetry breaking operator when $\Pi$ is a discrete series representation for $G/H$, using the idea of period integrals for the pair of reductive symmetric pairs.

The main result of this section is Theorem \ref{thm:24012120}, which provides a sufficient condition for the non-vanishing 
of period integrals in the general setting where $G \supset G'$ are {\it{arbitrary pairs}} of real reductive Lie groups.  

These results also apply to representations that are not necessarily tempered.

\medskip

\subsection{Discrete Series Representations for $X=G/H$}
~~~\newline
Let $(X, \mu)$ be a measure space and suppose that a group $G$ acts on $X$ in a measure-preserving fashion.  
Then, 
 there is a natural unitary representation of $G$ on the Hilbert space $L^2(X)$
 of square-integrable functions.  

An irreducible unitary representation $\Pi$ is called
 a {\it{discrete series representation}}
 for $X$, 
 if $\Pi$ can be realized in a closed subspace of $L^2(X)$.  
Let $\operatorname{Disc}(X)$
 denote the set of discrete series representations for $X$.  
Then $\operatorname{Disc}(X)$ is a (possibly, empty) subset of the unitary dual $\widehat G$ of $G$.

For $\Pi \in \widehat G$, 
 let $\Pi^{\vee}$ (resp.\ $\overline \Pi$)
 denote the contragredient (resp.\ complex conjugate) representation of $\Pi$. 
Then $\Pi^{\vee}$ and $\overline \Pi$ are unitarily equivalent representations.  
Moreover, 
 the set $\operatorname{Disc}(X)$ is closed under taking contragredient representations.

\subsection{Reductive Symmetric Spaces}
\label{subsec:symmsp}
~~~\newline
Let $G$ be a linear real reductive Lie group, $\sigma$ be an involutive automorphism of $G$, 
 and $H$ an open subgroup of $G^{\sigma}:=\{g \in G: \sigma g = g\}$.  
The homogeneous space $X=G/H$ is called a {\it{reductive symmetric space}}.

We take a Cartan involution $\theta$ of $G$
 that commutes with $\sigma$.  
Let $K$ be the corresponding maximal compact subgroup of $G$.  
Flensted-Jensen \cite{FJ80}
 and Matsuki--Oshima \cite{MO84} proved
 that $\operatorname{Disc}(G/H)\ne \emptyset$
 if and only if 
\begin{equation}
 \operatorname{rank} G/H=\operatorname{rank} K/H \cap K,
 \label{eqn:FJrank}
 \end{equation}
 generalizing 
 the Harish-Chandra rank condition \cite{HC66},
 $\operatorname{rank}G=\operatorname{rank} K$, 
 for the existence of discrete series representations of the group manifold $G$.

In contrast to Harish-Chandra’s discrete series representations for group manifolds, not every $\Pi \in \operatorname{Disc}(G/H)$ is tempered. There exist representations $\Pi \in \operatorname{Disc}(G/H)$ that are tempered (in the sense of Harish-Chandra) if and only if the centralizer $Z_G({\mathfrak{t}})$ is amenable, where $\mathfrak{t}$ is a Cartan subalgebra associated with the compact symmetric pair $(K, H \cap K)$. However, this condition does not imply that \emph{all} representations in $\operatorname{Disc}(G/H)$ are tempered (see \cite[Sect.~8.5]{BK21}).

Similarly, in contrast to Harish-Chandra's discrete series representations for group manifolds, 
 not every $\Pi \in \operatorname{Disc}(G/H)$ has
 a non-singular ${\mathfrak{Z}}({\mathfrak{g}}_{\mathbb{C}})$-infinitesimal
 character.  
This implies that if we realize the underlying $({\mathfrak{g}}, K)$-module $\Pi_K$
 via cohomological parabolic induction, 
 the parameter is not necessarily in \lq\lq{good range}\rq\rq\
 for this induction.  
In case the parameters are in the \lq\lq{good range}\rq\rq, 
 the minimal $K$-type  
 of $\Pi \in \operatorname{Disc}(G/H)$, 
 which we denote by $\mu(\Pi) \in \widehat K$, can be computed in a straight manner.  

\bigskip

\subsection{Period Integrals: Generalities}
\label{subsec:period}
~~~\newline
Let $X=G/H$ be a reductive symmetric space, 
 as in Section \ref{subsec:symmsp}.  
We now consider a pair $Y \subset X$
 of symmetric spaces as below.  
Suppose that $G'$ is a reductive subgroup of $G$, 
 stable under the involutions $\sigma$ and $\theta$ of $G$.  
Let $H':=H \cap G'$.  
Then $Y:=G'/H'$ is also a reductive symmetric space, 
 and there is a natural inclusion $\iota \colon Y \hookrightarrow X$, 
 which is $G'$-equivariant.

Let $\Pi$ be a discrete series representation for $X=G/H$.  
By convention, 
 we identify $\Pi$
 with its corresponding representation space in $L^2(X)$.  
Then, 
 the smooth representation $\Pi^{\infty} \in {\mathcal{M}}(G)$ is realized as a subspace of $(L^2 \cap C^{\infty})(X)$.

The first step is to prove the convergence of period integrals in the general setting where neither $\Pi \in \widehat{G}$ nor $\pi \in \widehat{G'}$ is assumed to be tempered.
\begin{theorem}[\cite{xks25}]
\label{thm:period}
For any $\Pi \in \operatorname{Disc}(X)$
 and any $\pi \in \operatorname{Disc}(Y)$, 
 the following period integral
\begin{equation}
\label{eqn:Bperiod}
   B \colon \Pi^{\infty} \times \pi^{\infty} \to {\mathbb{C}}, \quad
   (F, f) \mapsto \int_Y (\iota^{\ast} F)(y) f(y) d y
\end{equation}
 converges.  
Hence, 
 it defines a continuous $G'$-invariant bilinear form. 
In particular, 
the bilinear form \eqref{eqn:Bperiod} induces a symmetry breaking operator
\begin{equation}
\label{eqn:Tperiod}
   T_B \colon \Pi^{\infty} \to (\pi^{\vee})^{\infty}, \quad
   F \mapsto B(F, \cdot),  
\end{equation}
where $\pi^{\vee}$ denotes the contragredient representation of $\pi$. 
\end{theorem}

\medskip
The second step is to detect when the period integral $T_B$ does not vanish.
It should be noted that the period integral can vanish,
even when $\operatorname{Hom}_{G'}(\Pi^{\infty}|_{G'}, \pi^{\infty}) \neq 0$.   
This leads to the following question:

\begin{question}
\label{q:period}
Find a sufficient condition
 for the period integral \eqref{eqn:Bperiod} not to vanish.   
\end{question}

Some sufficient conditions have been derived in the special cases when both $X$ and $Y$ are group manifolds  \cite{HL, Va01},
 and when $X$ is a certain rank-one symmetric space \cite{OS21, OS25}.  
In a forthcoming paper \cite{xks25}, 
 we will prove the following theorem for the general pair of reductive Lie groups 
 $G' \subset G$
 and for their reductive symmetric spaces $Y \subset X$
 of higher rank:

\begin{theorem}[\cite{xks25}]
\label{thm:24012120}
Let $Y \subset X$ be as in the beginning of this section.  
Additionally, 
 we assume that $G$ is contained in a connected complex reductive Lie group $G_{\mathbb{C}}$
 and that $K$ and $K'$ are in the Harish-Chandra class.  
Let $\Pi \in \operatorname{Disc}(X)$
 and $\pi \in \operatorname{Disc}(Y)$
 both have non-singular infinitesimal characters.  
Suppose that the minimal $K$-types
 $\mu({\Pi}) \in \widehat K$
 and  $\mu'({\pi}) \in \widehat {K'}$ satisfy 
 the following two conditions:
\begin{equation}
\label{eqn:minK}
   [\mu({\Pi})|_{K'}: \mu'({\pi})] = 1;
\end{equation}
\begin{equation}
\label{eqn:241226}
\text{
a non-zero highest weight vector of $\mu(\Pi)$ 
 is contained in $\mu'(\pi)$.  }
\end{equation}
Then the period integral \eqref{eqn:Bperiod} is non-zero, 
 and consequently, 
 the corresponding symmetry breaking operator (SBO)
 in \eqref{eqn:Tperiod} is non-zero.  
\end{theorem}

\begin{remark}
In the case where $(G, G')=(GL(n,{\mathbb{R}}), GL(n-1,{\mathbb{R}}))$, 
 one of $K=O(n)$ or $K'=O(n-1)$
 is not in Harish-Chandra class.  
However, Theorem \ref{thm:24012120} holds in this case as well, 
provided that we define minimal $K$-types
 in terms of their irreducible ${\mathfrak{k}}$-summands.  
\end{remark}

\begin{remark}
\label{rem:K_mult_one}
Theorem~\ref{thm:24012120} applies to general pairs of real reductive Lie groups, $(G,G')$.
In the specific cases where 
\[(G,G')=(GL(n, \mathbb{R}), GL(n-1, \mathbb{R}))
\text{ or }
(U(p,q), U(p-1,q)),
\]
the assumption (\ref{eqn:minK}) is automatically derived from (\ref{eqn:241226}).
\end{remark}
\begin{remark}
\label{rem:24112401}
Yet another sufficient condition for the non-vanishing of the period integral \eqref{eqn:Bperiod} is 
\[
  \dim \invHom{K}{\mu(\Pi)}{C^{\infty}(K/M_H)}
  =
  \dim \invHom{K'}{\mu'(\pi)}{C^{\infty}(K'/M_H')}=1, 
\]
where $M_H$ is the centralizer of a generic element
 in ${\mathfrak{g}}^{-\theta, -\sigma}$
 in $H \cap K$, 
 and $M_H'$ is that of ${\mathfrak{g}'}^{-\theta, -\sigma}$
 in $H' \cap K'$.  
This condition is satisfied, 
 in particular, 
 when $K_{\mathbb{C}}/M_{H, {\mathbb{C}}}$
 and $K_{\mathbb{C}}'/M_{H, {\mathbb{C}}}'$
 are spherical.
However, the settings that we will treat in Sections \ref{sec:GHgeneral} and \ref{sec:upq2}
 are more general.  
\end{remark}

We give some examples of Theorem \ref{thm:period}
 in Sections \ref{sec:GHgeneral}
and \ref{sec:upq2} in settings where $X$ is a symmetric space of $G=GL(n,{\mathbb{R}})$
 and $G=U(p,q)$, 
 respectively.

\section{A Family of Representations of $GL(n, {\mathbb{R}})$}
In this section we introduce a family of irreducible unitary representations 
 of $GL(N, {\mathbb{R}})$ for $N=n$ or $n-1$
  that are not necessarily tempered, but are  discrete series representations of a symmetric space of $GL(n,{\mathbb{R}})$.  
\subsection{Weyl's Notation for $\widehat{O(N)}$}
\label{subsec:Weyl}
~~~\newline
We observe that the maximal compact subgroup $K = O(N)$ of $G = GL(N, {\mathbb{R}})$ is not of Harish-Chandra class when $N$ is even; that is, the adjoint action $\operatorname{Ad}(g)$ on $\mathfrak g$ is not always inner.
To discuss the branching laws for $(K, K')=(O(N), O(N{ - }1))$,
particularly those concerning the minimal $K$-types of discrete series representations of $G$ for reductive symmetric spaces $G/H$ (see Theorem~\ref{thm:24012120}),
we find that Weyl's notation (see \cite[Chap.\ V, Sect.\ 7]{Weyl97})---briefly recalled below---is more convenient and uniform than the conventional description based on highest weight theory.

Let $\widehat{O(N)}$ denote the set of equivalence classes of irreducible representations of $O(N)$.
Let $\Lambda^+(O(N))$ be the set 
 of $\lambda=(\lambda_1, \dots, \lambda_N) \in {\mathbb{Z}}^N$
 in one of the following forms.  
\begin{itemize}
\item[]Type I\,: $(\lambda_1,\cdots,\lambda_k,\underbrace{0,\cdots,0}_{N-k})$,
\item[]Type II\,: $(\lambda_1,\cdots,\lambda_k,\underbrace{1,\cdots,1}_{N-2k},\underbrace{0,\cdots,0}_k)$,
\end{itemize}
 where $\lambda_1\geq \lambda_2\geq\cdots\geq \lambda_k > 0$
 and $0\leq 2k \leq N$.

For any $\lambda \in \Lambda^+(O(N))$, 
 let $v_{\lambda}$ be the highest weight vector
 of the irreducible $U(N)$-module $\Kirredrep{U(N)}{\lambda}$.  
Then there exists a unique $O(N)$-irreducible submodule containing $v_{\lambda}$, 
 which we denote by 
 $\Kirredrep {O(N)}{\lambda}$.  
Weyl established the following bijection:
\begin{equation} 
\label{eqn:CWOn}
\Lambda^+(O(N)) \stackrel{\sim}{\longrightarrow} \widehat{O(N)},
\quad
\lambda \mapsto \Kirredrep {O(N)}{\lambda}.  
\end{equation}

\medskip
\subsection{Relative Discrete Series Representations of $GL(2,{\mathbb{R}})$}
~~~\newline
Let $\sigma_a$ ($a \in {\mathbb{N}}_+$) denote the relative discrete series representation of $GL(2,{\mathbb{R}})$
 with the following property:
\begin{alignat*}{2}
&\text{infinitesimal character $\frac 1 2(a,-a)$} 
&&\text{(Harish-Chandra parameter);}  
\\
&\text{minimal $K$-type $\Kirredrep{O(2)}{a+1, 0}$}
\quad
&&\text{(Blattner parameter)}.  
\end{alignat*}

We note that the restriction $\sigma_a|_{SL(2,{\mathbb{R}})}$ splits into the direct sum of a holomorphic (resp.\ anti-holomorphic) discrete series representation with minimal $K$-type ${\mathbb{C}}_{a+1}$
 (resp.\ ${\mathbb{C}}_{-(a+1)}$).

\subsection{Certain Family of (Non-Tempered) Irreducible Representations of $GL(n,{\mathbb{R}})$}
\label{subsec:PlG}
~~~\newline
Let $G=GL(n,{\mathbb{R}})$.  
For an integer $\ell$ with $0\leq 2 \ell \le n$, 
 let $P_{\ell}$ denote a real parabolic subgroup of $G$ whose Levi part is
 \[L_{\ell} := GL(2,{\mathbb{R}})^{\ell} \times GL(n-2\ell,{\mathbb{R}}).
 \]

For $\lambda=(\lambda_1, \dots, \lambda_{\ell}) \in {\mathbb{N}}_+^{\ell}$, 
 we define a unitary representation of $G$
 by means of normalized smooth parabolic induction: 
\begin{equation}
\label{eqn:PlG}
\Pi_{\ell}(\lambda):= \operatorname{Ind}_{P_{\ell}}^G \big(\bigotimes_{j=1}^{\ell} \sigma_{\lambda_j} \otimes {\bf{1}}\big).  
\end{equation}
Then $\Pi_{\ell}(\lambda)$ is an irreducible unitary representation of $GL(n,{\mathbb{R}})$ ({\it{cf}}.\ \cite{Vo86}).  
Moreover,
it is a tempered unitary representation if and only if $n=2 \ell + 1$ or $2 \ell$.  

For $2 k \le n-1$,  and for $\nu=(\nu_1, \dots, \nu_k)$, 
 we shall use an analogous notation $\pi_k(\nu)$
 for a family of irreducible unitary representations
 of $G'=GL(n-1,{\mathbb{R}})$.  

\bigskip

\subsection{Cohomological Parabolic Induction for $GL(n,{\mathbb{R}})$}
\label{subsec:qlG}
~~~\newline
An alternative construction of the representations $\Pi_\mathfrak{q}(\lambda )$
is given by cohomological induction.  

Let $2 \ell \le n$
 and ${\mathfrak{q}}_{\ell}$ be a $\theta$-stable parabolic subalgebra of ${\mathfrak{g}}_{\mathbb{C}} \simeq {\mathfrak{g l}}(n,{\mathbb{C}})$
 with the real Levi subgroup
\[
   L \equiv N_G({\mathfrak{q}}_{\ell}) \simeq ({\mathbb{C}}^{\times})^{\ell} \times GL(n-2 \ell,{\mathbb{R}}).  
\]

We set 
\[
   S_{\ell}:= \frac 1 2 \dim K/L=\ell(n-\ell-1).  
\]
Suppose that $\lambda =(\lambda_1, \dots, \lambda_{\ell})\in {\mathbb{Z}}^{\ell}$ satisfies
 $\lambda_1 > \cdots > \lambda_{\ell} >0$.  
We adopt a normalization such that the cohomological parabolic induction
 ${\mathcal{R}}_{{\mathfrak{q}}_{\ell}}^{S_{\ell}}({\mathbb{C}}_{\lambda})$
 has a ${\mathfrak{Z}}({\mathfrak{g}})$-infinitesimal character given by 
\begin{equation}
\label{eqn:GLinflmd}   
   \frac 1 2 (\lambda_{1}, \dots, \lambda_{\ell}, n-2\ell-1, \dots, 1+2 \ell-n, -\lambda_{\ell}, \dots, -\lambda_{1}) \in {\mathbb{C}}^n/{\mathfrak{S}}_n, 
\end{equation}
via the Harish-Chandra isomorphism.
Then its minimal $K$-type is given by
\begin{equation}
\label{eqn:23101406}
\mu_{\lambda}=(\lambda_{1}+1, \dots, \lambda_{\ell}+1, 0, \dots, 0) \in \Lambda^+(O(n))
\end{equation}
in Weyl's notation. 

The underlying $({\mathfrak{g}}, K)$-module of the $G$-modules $\Pi_{\ell}(\lambda)$  can be described  in terms of cohomological parabolic induction:  
\begin{equation}
\label{eqn:GLrealcomplex}    
  \Pi_{\ell}(\lambda)_K 
  \simeq 
  {\mathcal{R}}_{{\mathfrak{q}}_{\ell}}^{S_{\ell}}({\mathbb{C}}_{\lambda}). 
\end{equation}


If $n>2\ell$ then the $O(n)$-module $\Kirredrep{O(n)}{\mu_{\lambda}}$ stays irreducible 
when restricted to $SO(n)$, 
 and its highest weight is given by \linebreak
 $(\lambda_1+1, \dots, \lambda_{\ell}+1, \underbrace{0, \dots, 0}_{[\frac n 2]-\ell})$
 in the standard notation.  
If $n=2 \ell$, 
 then $\Kirredrep{O(n)}{\mu_{\lambda}}$ splits into the direct sum of two irreducible $SO(n)$-modules
 with highest weights $(\lambda_1+1, \dots, \lambda_{\ell-1}+1, \lambda_{\ell}+1)$
 and \linebreak 
 $(\lambda_1+1, \dots, \lambda_{\ell-1}+1, -\lambda_{\ell}-1)$.

The parameter $\lambda =(\lambda_1, \dots, \lambda_{\ell})$
 is in the \lq\lq{good range}\rq\rq\
 with respect to ${\mathfrak{q}}_{\ell}$
 in the sense of \cite{Vogan84}
 if the following condition is satisfied:

\begin{equation}
\label{eqn:good}
  \lambda_1 > \lambda_2> \cdots > \lambda_{\ell} > \max(n-2\ell-1, 0).  
\end{equation}

\medskip

\bigskip
\section{Restricting Discrete Series Representations
 for the Symmetric Space  $GL(n, {\mathbb{R}})/(GL(p, {\mathbb{R}}) \times GL(n-p, {\mathbb{R}}))$
 to the Subgroup $GL(n-1, \mathbb R)$}
\label{sec:GHgeneral}

In this section,
 we prove the existence of a non-zero $G'$-homomorphism from $\Pi^{\infty}$ to $\pi^{\infty}$, 
 where $\Pi \in \operatorname{Disc}(X)$ and $\pi \in \operatorname{Disc}(Y)$,
 by using the \lq\lq jumping fences\rq\rq\ trick in the translation theorems for symmetry breaking, 
 as explained in Section \ref{sec:tfsbo}.

Throughout this section, 
 we consider the following setup:
 $X=G/H$, $Y=G'/(H \cap G')$, 
 where $p+q=n$ and 
\begin{align}
\label{eqn:npq}
 (G, H)&=(GL(n,{\mathbb{R}}), GL(p, {\mathbb{R}}) \times GL(q, {\mathbb{R}})),
 \\
 (G', H')&=(GL(n-1,{\mathbb{R}}), GL(p, {\mathbb{R}}) \times GL(q-1, {\mathbb{R}})).
\end{align}

The first two subsections focus
 on describing  $\operatorname{Disc}(K/H \cap K)$
 and $\operatorname{Disc}(G/H)$.  
We then apply Theorem \ref{thm:24012120}
 to prove the non-vanishing of the period integral under the assumption on the minimal $K$-types, 
 as described in \eqref{eqn:241226}.  
We shall see that the parity condition allows us to \lq\lq{jump the fences}\rq\rq\
 for this interlacing pattern by iteratively applying Theorems \ref{thm:23081404} and \ref{thm:24102601}.  
This leads to the whole range of parameters $(\lambda, \nu)$ for the non-vanishing of symmetry breaking in the restriction $G\downarrow G'$, 
 as detailed in Theorem \ref{thm:12}.

\bigskip

\subsection{Description of $\operatorname{Disc}(K/H \cap K)$}
~~~\newline
In the setting \eqref{eqn:npq}, 
 the pair of maximal compact subgroups $(K, H \cap K)$ is given by $(O(p+q), O(p) \times O(q))$.  
The following result extends the Cartan--Helgason theorem, which was originally formulated for connected groups, 
 to the case of disconnected groups.  

\begin{proposition}
\label{prop:23101419} 
Let $\ell:=\min(p,q)$.
In Weyl's notation (see Section \ref{subsec:Weyl}), 
 $\operatorname{Disc}(O(p+q)/O(p) \times O(q))$ is given by
\[
\{
  \Kirredrep{O(p+q)}{\mu}: \mu=(\mu_1, \dots, \mu_\ell,\overbrace{0, \dots,0}^{\max(p,q)}) \in (2{\mathbb{Z}})^{p+q}, 
  \mu_1 \ge \cdots \ge \mu_\ell \ge 0
\}.  
\]
\end{proposition}

If $p\neq q$, or if $\mu_\ell=0$, 
 then the $O(p+q)$-module $\Kirredrep{O(p+q)}{\mu}$ remains irreducible 
 when restricted to $SO(p+q)$.  
If $p=q$ and $\mu_\ell \ne 0$, 
 then $\Kirredrep{O(p+q)}{\mu}$ decomposes into the direct sum 
 of two irreducible $SO(p+q)$-modules.  

\bigskip
\subsection{Discrete Series for $GL(p+q, {\mathbb{R}})/(GL(p, {\mathbb{R}}) \times GL(q, {\mathbb{R}}))$}
~~~\newline
In this subsection, 
 we provide a complete description of discrete series representations for $G/H$ in the setting of \eqref{eqn:npq}.  

Let $G=GL(n, \mathbb{R})$ and $0 \le 2 \ell \le n$.
As recalled in \eqref{eqn:PlG} and \eqref{eqn:GLrealcomplex}, $\Pi_\ell(\lambda)$ denotes the irreducible unitary representation of $G$ obtained via parabolic induction, or equivalently,
through cohomological parabolic induction.

\medskip

\begin{proposition}
\label{prop:DiscGH}
Let $n=p+q$ and $\ell :=\min(p, q)$.  
Then the set of discrete series representations for $G/H$ is given by 
\[
  \{\Pi_{\ell}(\lambda):
\lambda=(\lambda_1, \dots, \lambda_\ell) \in (2{\mathbb{Z}}+1)^{\ell},
\lambda_1 > \lambda_2> \cdots >\lambda_\ell>0
\}.  
\]
\end{proposition}

\medskip
The ${\mathfrak{Z}}({\mathfrak{g}}_{\mathbb{C}})$-infinitesimal character
 of the $G$-module $\Pi_{\ell}(\lambda)$ is non-singular 
if \eqref{eqn:good} holds, 
 or equivalently, 
 if $\lambda_\ell > n-2\ell -1$.

To verify Proposition \ref{prop:DiscGH}, we make use of Matsuki--Oshima's description \cite{MO84} of discrete series representations, which may vanish, 
 along with a detailed computation of cohomological parabolic induction
beyond \lq\lq{good range}\rq\rq, 
 specifically, 
 when $\lambda_{\ell} \le n - 2\ell -1$
 as in the similar case thoroughly studied in \cite{Kobayashi92}.  
We note that for such a singular parameter $\lambda$,  
 neither the irreducibility nor the non-vanishing 
 of cohomological parabolic induction is guaranteed by the general theory \cite{Vogan84}.  
However, 
 it turns out that both non-vanishing and irreducibility do hold in our specific setting.

\bigskip

We also derive an explicit formula for the minimal $K$-type $\mu(\Pi_{\ell})$
 of the $G$-module $\Pi_{\ell}(\lambda)$:
 it is given in Weyl's notation as follows.  
\[
    \mu(\Pi_{\ell}(\lambda))=\Kirredrep{O(n)}{\lambda_1+1, \dots, \lambda_{\ell}+1, 0, \dots, 0}.  
\]

\medskip
\subsection{Comparison of Minimal $K$-types for Two Groups $G' \subset G$}
~~~\newline
Let $n=p+q$.
We realize $H=GL(p, {\mathbb{R}}) \times GL(q, {\mathbb{R}})$
 in the standard block-diagonal form as a subgroup of $G=GL(n, {\mathbb{R}})$,
 and we realize $G'=GL(n-1, {\mathbb{R}})$
as a subgroup of $G$, corresponding to the partition $n= (n-1)+1$.  
Accordingly, 
 we obtain an embedding of the reductive symmetric space 
\[
Y= GL(n-1, {\mathbb{R}})\big/\bigl(GL(p, {\mathbb{R}}) \times GL(q-1, {\mathbb{R}})\bigr)
\] 
 of $G'$ into $X=G/H$.

We recall from Proposition~\ref{prop:DiscGH}
 that any discrete series representation for $X$ with a non-singular ${\mathfrak{Z}}({\mathfrak{g}}_{\mathbb{C}})$-infinitesimal character is of the form $\Pi_{\ell}(\lambda)$, 
 where $\ell =\min(p,q)$ and $\lambda=(\lambda_1, \dots, \lambda_{\ell}) \in (2{\mathbb{Z}}+1)^{\ell}$, 
 satisfying the regularity condition \eqref{eqn:good}.  

\medskip
 We now assume that $2 p \le n -1$, where $n=p+q$.  
In this case, $\ell=\min(p,q)=\min(p,q-1)$.
Let $\pi_\ell(\nu)$ be the irreducible unitary representation of $G'=GL(n-1, \mathbb R)$, as defined in the same way as $\Pi_\ell(\lambda)$ for $G=GL(n, \mathbb R)$. Then, the set of discrete series representations for the smaller symmetric space:
\[
 Y=G'/H'=GL(n-1,\mathbb R)\big/\big(GL(p, \mathbb R) \times GL(q-1, \mathbb R)\big)
 \]
is given by 
\[
  \left\{\pi_{\ell}(\nu):
\nu=(\nu_1, \dots, \nu_\ell) \in (2{\mathbb{Z}}+1)^{\ell},
\nu_1 > \nu_2> \cdots >\nu_\ell>0
\right\}.  
\]

With these preparations, we apply Theorems~\ref{thm:period} and~\ref{thm:24012120}
 to the pair $(\Pi_{\ell}(\lambda), \pi_{\ell}(\nu)) \in \widehat{G} \times \widehat{G'}$.  
The assumption \eqref{eqn:minK} on minimal $K$-types is automatically satisfied for the pair $(K,K')=(O(n), O(n-1))$, 
 whereas the condition \eqref{eqn:241226} is computed explicitly as follows.

\begin{lemma}
\label{lem:23102202}
The condition \eqref{eqn:241226} holds 
 if and only if
\begin{equation}
\label{eqn:24122605a}
  \lambda_1 = \nu_1 > \lambda_2 = \nu_2 > \cdots > \lambda_ \ell = \nu_\ell>0.  
\end{equation}
\end{lemma}

By Theorem \ref{thm:24012120}, 
 we obtain the following.  

\begin{proposition}
\label{prop:24012302}
Suppose $2 \ell \le n-1$.
Then we have
\[
\dim \invHom{G'}{\Pi_{\ell}(\lambda)^{\infty}|_{G'}}{\pi_\ell(\nu)^{\infty}} =1
\]
 for any $\lambda \in (2 {\mathbb{Z}}+1)^{\ell}$
 and $\nu \in (2 {\mathbb{Z}}+1)^{\ell}$
 satisfying \eqref{eqn:24122605a}. 
\end{proposition}
\begin{remark}
\label{rem:250209}
Alternatively, we can prove Proposition~\ref{prop:24012302} by relying on the isomorphism  (\ref{eqn:GLrealcomplex}) and Mackey theory.
To do this, we use the fact that the $G'$-action on the generalized real flag manifold $G/P_\ell$ has an open dense orbit, and that the isotropy subgroup is contained in $P'_\ell$,
which is a parabolic subgroup of $G$ of the same type.
\end{remark}

\medskip
\subsection{Jumping Fences}
\label{subsec:jump}
~~~\newline
In this section, 
 we analyze a phenomenon in which a certain parity condition allows us to \lq\lq{jump the fence}\rq\rq, 
 of the interleaving pattern 
 in Theorems~\ref{thm:23081404} and~\ref{thm:24102601}.  
We discover that this phenomenon indeed occurs for some geometric settings in the context of symmetry breaking for $GL(n, {\mathbb{R}}) \downarrow GL(n-1, {\mathbb{R}})$.  
As a result, we provide a refinement of the (non-)vanishing results of symmetry breaking.

We begin with the setting where
$\Pi_{\ell}(\lambda)$ are irreducible unitary representations of $G$,
 and $\pi_{k}(\nu)$ are those of $G'$, 
with $0 \le 2\ell \le n$ and $0 \le 2 k \le n-1$, 
 as introduced in Sections \ref{subsec:PlG} and \ref{subsec:qlG}.  
In this generality, 
 we impose a slightly stronger
 than the good range condition \eqref{eqn:good}, 
 that is, 
 the following condition on the parameter 
 $\lambda_1$, $\dots$, $\lambda_{\ell}$: 
\begin{equation}
\label{eqn:sgood}
 \lambda_1 > \lambda_2> \cdots > \lambda_{\ell}
 >\max(n-2\ell-1, n-2k-3, 0).  
\end{equation}

\begin{remark}
\label{rem:24111606}
For the application of Corollary~\ref{cor:241120Osaka}  to Theorem~\ref{thm:12}, we use the case where $\ell = k$.
In this case, or more generally, if $\ell \le k+1$, 
 the condition \eqref{eqn:sgood} reduces to the good range condition \eqref{eqn:good}.  
\end{remark}

\begin{corollary}
\label{cor:241120Osaka}
Let $\nu \in (2 {\mathbb{Z}}+1)^{k}$
 satisfying 
\begin{equation}
\label{eqn:nugood}
 \nu_1 >\nu_2 > \cdots > \nu_k > \max(0, n-2 k -2).  
\end{equation}
Then the following two conditions are equivalent:

{\rm{(i)}}\enspace
there exists $\lambda \in (2 {\mathbb{Z}}+1)^\ell$
 satisfying \eqref{eqn:sgood}
 such that 
\[
  \invHom{G'}{\Pi_{\ell}(\lambda)^{\infty}|_{G'}}{\pi_k(\nu)^{\infty}} \ne \{0\}; 
\] 

{\rm{(ii)}}\enspace
for every $\lambda \in (2 {\mathbb{Z}}+1)^{\ell}$ satisfying \eqref{eqn:sgood}, 
one has 
\[
\invHom{G'}{\Pi_{\ell}(\lambda)^{\infty}|_{G'}}{\pi_k(\nu)^{\infty}} \ne \{0\}.  
\]
\end{corollary}

Thus, 
 Corollary \ref{cor:241120Osaka} allows us 
 to tear down all the \lq\lq{fences}\rq\rq\
 of the weakly interleaving pattern given by Lemma \ref{lem:23102202}, 
resulting in the following result:
\begin{theorem}
\label{thm:12}
Suppose $2 \ell < n$.
Then 
\begin{equation}
\label{eqn:mainbranch}
\dim \invHom {G'}{\Pi_{\ell}(\lambda)^{\infty}|_{G'}}{\pi_\ell(\nu)^{\infty}}=1
\end{equation}
for any $\lambda, \nu \in (2{\mathbb{Z}}+1)^{\ell}$
 satisfying the regularity conditions:
\begin{align*}
&\lambda_1 > \lambda_2> \cdots > \lambda_{\ell} > n-2 \ell -1,
\\
&\nu_1 > \nu_2> \cdots > \nu_\ell > n-2\ell -1.
\end{align*}
\end{theorem}

We already know that the left-hand side of \eqref{eqn:mainbranch}
is either 0 or 1, according to the multiplicity-freeness theorem \cite{xsunzhu}
 for $GL(n,{\mathbb{R}}) \downarrow GL(n-1, {\mathbb{R}})$,
since $\Pi_{\ell}(\lambda)$ and $\pi_\ell(\nu)$ are irreducible
 as $G$- and $G'$-modules, 
 respectively.  
Our claim is that the multiplicity is non-zero, as a consequence of \lq\lq{jumping all the fences}\rq\rq.


\bigskip

\section{Restricting Discrete Series Representations for the Symmetric Spaces $U(p,q)/(U(r, s) \times U(p-r, q-s))$ to 
the Subgroup $U(p-1, q)$}
\label{sec:upq2}
In this section, 
 we revisit the case  where 
 \[(G, G')=(U(p,q), U(p-1, q)),
 \]
 and discuss the branching of the restriction $\Pi|_{G'}$, 
 where $\Pi$ is a {\it{non-tempered}} irreducible representation of $G$.  
Specifically, 
 we consider a discrete series representation $\Pi$ for the symmetric space
\[
  G/H=U(p,q)/(U(r, s) \times U(p-r, q-s)),
\]
and prove that $\invHom{G'}{\Pi^{\infty}|_{G'}}{\pi^{\infty}} \ne 0$
 for some family of irreducible representations $\pi \in \widehat{G'}$,
 which are not necessarily tempered. 
 
 The irreducible unitary representations $\pi$ of the subgroup $G'$ for which $\invHom{G'}{\Pi^{\infty}|_{G'}}{\pi^{\infty}} \ne 0$ were completely determined when $(r,s)=(0,1)$, as a particular case of \cite[Thm.\ 3.4]{xkpja93}, 
 which corresponds to the discretely decomposable case. 
 In the case where $\pi$ occurs as a discrete series representation for a symmetric space $G'/H'$, a non-vanishing result was recently proven in \cite{OS25}
 when $(r,s)=(1,0)$.  
 
We provide a non-vanishing theorem in Theorem \ref{thm:250126FJ}
for the general case of $(p,q,r,s)$
under a certain interleaving condition on the parameters.
Our proof again utilizes the non-vanishing theorem of the period integral
for specific parameters, 
 as stated in Theorem \ref{thm:24012120}, 
 as well as the non-vanishing result of symmetry breaking under translations inside \lq\lq{fences}\rq\rq, 
 as stated in Theorem \ref{thm:24120703}.  

\medskip
\subsection{A Family of (Non-Tempered) Irreducible Unitary Representations of $U(p,q)$}
~~~\newline
In this subsection, 
 we define a family of irreducible unitary representations of $G=U(p,q)$.  
In the next subsection, 
 Proposition \ref{prop:GHupq} shows that
 that any discrete series representation for the symmetric space
 $X=U(p, q)/(U(r, s) \times U(p-r, q-s))$
 is of this form when $2 r \le p$ and $2s \le q$.

Let ${\mathfrak{j}}$ be a compact Cartan subalgebra,
$\{H_1, \dots, H_{p+q}\}$ be the standard basis $\sqrt{-1}{\mathfrak{j}}$, 
 and $\{f_1, \dots, f_{p+q}\}$ its dual basis.  
We fix a positive system of 
$
   \Delta({\mathfrak{k}}_{\mathbb{C}}, {\mathfrak{j}}_{\mathbb{C}})
$ by defining
\[
  \Delta^+({\mathfrak{k}}_{\mathbb{C}}, {\mathfrak{j}}_{\mathbb{C}})
 =\{f_i-f_j: 1 \le i < j \le p\text{ or }p+1 \le i < j \le p+q\}.  
\]

Given $Z=(z_1, \dots, z_{p+q}) \in \sqrt{-1}{\mathfrak{j}} \simeq {\mathbb{R}}^{p+q}$, 
 we define a $\theta$-stable parabolic subalgebra $\mathfrak{q}\equiv {\mathfrak{q}}(Z)={\mathfrak{l}}+{\mathfrak{u}}$
 of ${\mathfrak{g}}_{\mathbb{C}}={\mathfrak{gl}}(p+q, {\mathbb{C}})$
 such that the set of weights of the unipotent radical $\mathfrak{u}$ is given by
\[
  \Delta({\mathfrak{u}}, {\mathfrak{j}}_{\mathbb{C}})
  =\{\alpha \in \Delta({\mathfrak{g}}_{\mathbb{C}}, {\mathfrak{j}}_{\mathbb{C}}):
\alpha(Z)>0\}.  
\]
Any $\theta$-stable parabolic subalgebra of ${\mathfrak{g}}_{\mathbb{C}}$ is $K$-conjugate to ${\mathfrak{q}}(Z)$ for some $Z \in {\mathbb{R}}_{\ge}^p \times {\mathbb{R}}_{\ge}^q$.  
We are particularly interested in the following:
\begin{setting}
\label{set:rs}
Let $0 \le 2 r \le p$, 
$0 \le 2s \le q$, 
 and 
\begin{equation}
\label{eqn:xyz}
Z=(x_1, \dots, x_r, 0^{p-2r}, -x_r, \dots, -x_1;y_1, \dots, y_s, 0^{q-2s}, -y_s, \dots, -y_1)
\end{equation}
with $x_1 > \cdots>x_r >0$, 
$y_1>\cdots, y_s >0$, 
and $x_i \ne y_j$ for any $i$, $j$.  
\end{setting}
 
In this case, the (real) Levi subgroup $L$, 
 the normalizer of the $\theta$-stable parabolic subalgebra ${\mathfrak{q}}(Z)$ in $G$, 
 depends only on $r$ and $s$, 
 and is given by 
\begin{equation}
\label{eqn:Lpq}
   L \equiv L^U_{p,q;r,s} \simeq {\mathbb{T}}^{2r+2s} \times U(p-2r, q-2s).  
\end{equation}

\begin{lemma}
\label{lem:250129}
Let $G=U(p,q)$. 
We fix $r$ and $s$ such that
 $2r \le p$ and $2s \le q$.  
Then, there is a one-to-one correspondence
 among the following three objects:
\newline
{\rm{(i)}}\enspace
$\theta$-stable parabolic subalgebras $\mathfrak{q}\equiv{\mathfrak{q}}(Z)$, 
 where $Z$ is of the form as given in Setting \ref{set:rs}.  
\newline
{\rm{(ii)}}\enspace
Interleaving patterns $D \in {\mathfrak{P}}({\mathbb{R}}^{r,s})$
in ${\mathbb{R}}_{>}^r \times {\mathbb{R}}_{>}^s$.  
\newline
{\rm{(iii)}}\enspace
Data $\kappa=\{(r_j), (s_j), M\}$
 with $1 \le M \le \min(r,s)$ and 
\begin{equation}
\label{eqn:rs}
  0 \le r_1 < \cdots<r_{M-1}<r_M=r, 
\,\,
  0 < s_1 < \cdots <s_{M-1}\le s_M=s.  
\end{equation}
\end{lemma}

\begin{remark}
We allow the cases $r_1=0$ or $s_{M-1}=s_M$, 
 but assume that $s_1>0$ and $r_{M-1}< r_M$.  
\end{remark}

\begin{proof}
We describe the natural morphisms, 
which establish the one-to-one correspondence among (i), (ii) and (iii).  
\newline
(i) $\Leftrightarrow$ (ii)\enspace
By definition, 
 an interleaving pattern $D$ in ${\mathfrak{P}}({\mathbb{R}}^{r,s})$ defines
 a $\theta$-stable parabolic subalgebra ${\mathfrak{q}}(Z)$ via \eqref{eqn:xyz}.  
Conversely, 
 it is clear that the $\theta$-stable parabolic subalgebra ${\mathfrak{q}}(Z)$
 associated with $Z$ in Setting \ref{set:rs} depends solely on the interleaving pattern of $x$, $y$ in ${\mathbb{R}}_{>}^r \times {\mathbb{R}}_{>}^s$.  
\newline
(ii) $\Leftrightarrow$ (iii)\enspace
Given a condition in \eqref{eqn:rs}, 
 we associate the following interleaving pattern $D$ 
in ${\mathbb{R}}_{>}^r \times {\mathbb{R}}_{>}^s$
 defined by 
\begin{multline}
\label{eqn:xypartition}
 x_1 > \cdots>x_{r_1} > y_1>\cdots> y_{s_1} >x_{r_1+1}> \cdots>x_{r_2}>y_{s_1+1}> \cdots
\\
 \cdots > y_{s_{M-1}}> x_{r_{M-1}+1}> \cdots > x_{r_M}> y_{s_{M-1}+1}>\cdots >y_{s_M},   
\end{multline}
and vice versa.  
\end{proof}

Let $D$ be an interleaving pattern in ${\mathbb{R}}_{>}^r \times {\mathbb{R}}_{>}^s$ as in \eqref{eqn:xypartition}.
For $A \in {\mathbb{R}}$, 
 we set 
\[
  D_{>A}:=\{(x,y) \in D:x_i>A, y_j>A\text{ for any $i$, $j$}\}.  
\]

Suppose that $L={\mathbb{T}}^{2r+2s} \times U(p-2r, q-2s)$
 is the real Levi subgroup
 for the $\theta$-stable parabolic subalgebra ${\mathfrak{q}}$, 
 which is associated to an interleaving pattern
 $D \in {\mathfrak{P}}({\mathbb{R}}^{r,s})$ in Lemma \ref{lem:250129}.  
For $\lambda =(x,y) \in ({\mathbb{Z}}+\frac{p+q-1}2)^{r+s}$, 
 we define a one-dimensional representation of the double covering group of the torus ${\mathbb{T}}^{2(r+s)}$,
  to be denoted by ${\mathbb{C}}_{\widetilde\lambda}$,
such that its differential is given by the formula \eqref{eqn:xyz}.  
We extend it to a one-dimensional representation of ${\mathfrak{q}}={\mathfrak{l}}+{\mathfrak{u}}$, 
 by letting ${\mathfrak{u}}(p-2r, q-2s)+{\mathfrak{u}}$ act trivially.  
The character ${\mathbb{C}}_{\widetilde\lambda}$ 
 is in the fair range (respectively, in the good range) with respect to ${\mathfrak{q}}$ in the sense of \cite{Vogan84}, 
 if $\lambda \in D_{>0}$
 (respectively, $\lambda \in D_{>Q}$), 
 where we set
\[ 
  Q:=\frac 1 2 (p+q-1)-r-s.  
\] 

When $\lambda \in D_{>0}$, cohomological parabolic induction gives a unitarizable $({\mathfrak{g}}, K)$-module, 
 which is possibly zero (\cite{Vogan84}). 
It is irreducible if non-zero.  
Let $\Pi_{\lambda}$ denote the unitarization.  
The unitary representation $\Pi_{\lambda}$ is non-tempered 
 if $p \ne 2r$ and $q\ne 2s$.

In our normalization, 
 the ${\mathfrak{Z}}({\mathfrak{g}}_{\mathbb{C}})$-infinitesimal character of the $G$-module $\Pi_{\lambda}$ is given by 
\[
   \text{\eqref{eqn:xyz}} \oplus (Q,Q-1, \dots, 1-Q, -Q) \in {\mathbb{C}}^{p+q}/{\mathfrak{S}}_{p+q}.  
\]
When $D_{>Q}$, 
the general theory guarantees that $\Pi_{\lambda}$ is non-zero
and that the highest weight of its minimal $K$-type is given as follows:

\begin{alignat*}{2}
  (\mu_{\lambda})_i&=-(\mu_{\lambda})_{p+1-i}= \lambda_i+ \frac{-p+q+1}{2}+\ell_i \quad 
&&\text{for $1 \le i \le r$, }
\\
  (\mu_{\lambda})_{p+i}&=-(\mu_{\lambda})_{p+q+1-i}
 = \lambda_{r+i}+ \frac{p-q+1}{2}-\ell_{r+i} \quad 
&&\text{for $1 \le i \le s$, }  
\\
  (\mu_{\lambda})_{i}&=0\qquad 
&&\text{otherwise.}
\end{alignat*}

Here, we define $\ell_i\equiv \ell_i (D)\in {\mathbb{Z}}$
for $1 \le i \le r+s$, 
 depending on the interleaving pattern $D$, by
\begin{alignat*}{2}
\ell_i(D):=&\sharp \{x_k: x_k > x_i\}-\sharp \{y_k: y_k > x_i\}
\qquad
\text{for $1 \le i \le r$, }
\\
\ell_{r+i}(D):=&\sharp \{x_k: x_k > y_i\}-\sharp \{y_k: y_k > y_i\}
\qquad
\text{for $1 \le i \le s$.}
\end{alignat*}

\begin{example}
Let $(r, s)=(3,2)$
 and $D=\{x_1 > y_1 > y_2 >x_2>x_3\}$.  
Then 
\[
  \ell_1(D)=0,\,\ell_2(D)=-1,\, \ell_3(D)=0;
  \ell_4(D)=1,\, \ell_5(D)=0. 
\]
\end{example}

\medskip
\subsection{Discrete Series Representations for the Symmetric Space $U(p,q)/(U(r,s) \times U(p-r, q-s))$}
~~~\newline
Let $H=U(p_1,q_1) \times U(p_2,q_2)$
 be a natural subgroup of $G=U(p,q)$, 
 where $p_1 + p_2=p$ and $q_1+ q_2=q$.  
The symmetric space $G/H$ has a discrete series representation
 if and only if the rank condition (\ref{eqn:FJrank}) holds, that is, 
\begin{equation}
\label{eqn:rankFJ}
  \min(p_1, p_2) +\min(q_1, q_2) = \min(p_1 + q_1, p_2 + q_2).  
\end{equation} 
{}From now on, 
 without loss of generality, 
 we assume 
that 
\[
   H=U(r, s) \times U(p-r, q-s)\qquad
  \text{with $2 r \le p$ and $2 s \le q$}.
\]  
Discrete series representations for a reductive symmetric space $G/H$ are decomposed into families corresponding to $H^d$-closed orbits on the real flag variety of $G^d$. Here, 
$(G^d,H^d)$ is the dual symmetric pair of $(G,H)$, see \cite{FJ80, MO84}.
In the above setting, we have
$$
(G^d, H^d)=(U(r+s,p+q-r-s),U(r,p-r) \times U(s, q-s)),
$$ and 
there are $\frac{(r+s)!}{r! s!}$ closed orbits of the subgroup $H^d$ on the real flag variety of $G^d$. 
These orbits are parametrized by 
interleaving patterns
$\mathfrak{P}({\mathbb{R}}^{r,s})$
in ${\mathbb{R}}_>^r \times {\mathbb{R}}_>^s$.
 
\begin{proposition}
\label{prop:GHupq}
Suppose $0 \le 2r \le p$ and $0 \le 2s \le q$.  
Then the set of discrete series representations
\[
  \operatorname{Disc}\!\bigl(U(p,q)/(U(r,s) \times U(p-r, q-s))\bigr)
\]
can be described as the disjoint union
\[
  \coprod_{D \in {\mathfrak{P}}(\mathbb{R}^{r,s})}
  \bigl\{\Pi_{\lambda} : \lambda \in D_{>0} \cap (\mathbb{Z} + \tfrac{p+q-1}{2})^{\,r+s}\bigr\}.
\]
\end{proposition}

As mentioned in the previous subsection,  
 $\Pi_{\lambda}$ may vanish if $\mathbb{C}_{\tilde{\lambda}}$ is not in the good range, specifically, if $\lambda \in D_{>0} \setminus D_{>Q}$. 
The condition for the non-vanishing of $\Pi_{\lambda}$
 involves a number of inequalities of $\lambda$ that depend heavily on $D\in \mathfrak{P}({\mathbb{R}}^{r,s})$
 (see \cite[Chap.\ 5]{Kobayashi92}).

\medskip 

\subsection{Branching for $U(p,q) \downarrow U(p-1,q)$}
~~~\newline
We are ready to state our main results of this section. 
\begin{theorem}
\label{thm:250126FJ}
Suppose that $0 \le 2r \le p-1$, 
 $0 \le 2s \le q$
and $D, D' \in {\mathfrak{P}}({\mathbb{R}}^{r,s})$.
 Let ${\mathfrak{q}}$ be the $\theta$-stable parabolic subalgebra of ${\mathfrak{g}}_{\mathbb{C}}$
 and $(\{r_i\}, \{s_i\}, M)$ be the data, associated with $D$,
 as in Lemma \ref{lem:250129}.  
Similarly, 
 let ${\mathfrak{q}}'$ be the $\theta$-stable parabolic subalgebra of ${\mathfrak{g}}_{\mathbb{C}}'$
 and $(\{r_j'\}, \{s_j'\}, M')$ be the data,
 associated with $D'$.   
 We set
$Q=\frac 1 2 (p+q-1)-r-s$ and $Q'=Q-\frac{1}{2}$.

Assume that $D=D'$, or equivalently that
$M'=M$, 
$r_i'=r_i$ $(1 \le i \le M)$ and $s_i'=s_i$ $(1 \le i \le M)$.  
Then we have the following identity:
\begin{equation}
\label{eqn:GHupq1}
  \dim \invHom{G'}{\Pi_{\lambda}^{\infty}|_{G'}}{\pi_{\nu}^{\infty}}=1, 
\end{equation}
 if $\lambda =(x,y) \in D_{>Q} \cap ({\mathbb{Z}}+ Q)^{r+s}$
 and $\nu =(\xi, \eta)\in D'_{>Q'} \cap ({\mathbb{Z}}+Q')^{r'+s}$ 
 satisfy the following interleaving pattern:
 \begin{multline*}
 x_1 >\xi_1>\dots
 >x_{r_1}>\xi_{r_1}>\eta_1>y_1>\dots>\eta_{s_1}>y_{s_1}>
 \\
>x_{r_1 + 1} >\xi_{r_1+1} > \dots > x_{r_2}>\xi_{r_2}>\eta_{s_1+1}>y_{s_1+1}>\dots>\eta_{s_2}>y_{s_2}>
\\
\dots > x_{r_M}>\xi_{r_M}>\eta_{s_{M-1}+1}>y_{s_{M-1}+1}>\dots>\eta_{s_M}>y_{s_M}.
 \end{multline*}
 \end{theorem}

\begin{remark}
The interleaving pattern on $\lambda=(x,y)$ and $\nu=(\xi, \eta)$
 in Theorem \ref{thm:250126FJ} is equivalent to 
 that $[D D' +] \in {\mathfrak{P}}({\mathbb{R}}^{r+1,s}) \times {\mathfrak{P}}({\mathbb{R}}^{r,s})$ is a \emph{coherent pair}, 
 where
 $[D D' +]$ is an interleaving pattern of $(\lambda, x_{r+1}, \nu)$, defined by the inequalities $D$ for the entries of $\lambda$ and $\nu$, along with the condition that $x_{r+1}$ is smaller than any of the entries of $\lambda$ and $\nu$.
For various equivalent definitions of \lq\lq coherent pairs", 
 we refer to \cite{HKS}.
\end{remark}

\medskip
By the stability theorem for multiplicities
in symmetry breaking within fences (Theorem~\ref{thm:24120703}), 
the proof of Theorem~\ref{thm:250126FJ} reduces to the following proposition:
\begin{proposition}
\label{prop:250128}
In the setting and assumptions of Theorem \ref{thm:250126FJ}, 
 the equality \eqref{eqn:GHupq1} holds
 if $\lambda=(x,y) \in \mathbb{R}^{r+s}$ and $\nu=(\xi,\eta)  \in \mathbb{R}^{r'+s}$ satisfy following conditions:
\begin{equation}
\label{eqn:25012550}
  \begin{cases}
  x_i=\xi_i + \frac 1 2 \qquad &(1 \le i \le r), 
\\
  y_i =\eta_i-\frac 1 2 &(1 \le i \le s).  
\end{cases}
\end{equation}
\end{proposition}

\begin{proof}
We apply the non-vanishing theorem for the period integral of discrete series representations
(Theorem~\ref{thm:24012120})
to the symmetric spaces $X=G/H$ and $Y=G'/H'$, where
$G' = U(p-1,q)$ is realized as a subgroup of $G$ such that
$H' := H \cap G' \simeq U(r, s) \times U(p-r-1, q-s)$.
Condition \eqref{eqn:25012550} then ensures that the assumption \eqref{eqn:241226} 
on minimal $K$-types in Theorem~\ref{thm:24012120} is satisfied,
while \eqref{eqn:minK} is immediate.
\end{proof}

By Theorem~\ref{thm:24120703}, the result in Proposition~\ref{prop:250128} 
extends to all the parameters stated in Theorem~\ref{thm:250126FJ} via translations within the initial fences.  
Thus, the non-vanishing of symmetry breaking is established, completing the proof of Theorem~\ref{thm:250126FJ}.

\medskip
In contrast to the $GL(n,{\mathbb{R}})$ case in Section \ref{subsec:jump}, 
we note that jumping the fences is prohibited in the $U(p,q)$ case due to a different parity condition.  
\bigskip

\section{Remarks and Examples}
 \label{sec:Arthur}

In this section, we make some general remarks and illustrate our results with examples of tempered and non-tempered representations, including extensions to limits of discrete series representations.

\subsection{Geometric Observations} 
Following \cite[Def.~3.1]{xk22_JA}, we recall the generalized notions of \lq\lq{Borel subalgebras}\rq\rq\
 (\emph{relative Borel subalgebras})
 and complex Levi subalgebras for reductive symmetric spaces $G/H$ associated with involutive automorphisms $\sigma$ of $G$. 
 These notions were used to refine a generalization of Casselman's embedding theorem {\cite{xk22}} to representations with $H$-distinguished vectors.

Let $G_U$ be a maximal compact subgroup of $G_{\mathbb{C}}$, chosen so that $G_U \cap G$
 and $G_U \cap H$ are also maximal compact subgroups of $G$ and $H$, 
 respectively.  
We fix an $\operatorname{Ad}(G)$-invariant, 
 non-degenerate symmetric bilinear form
 on the Lie algebra ${\mathfrak{g}}$, 
 which is also non-degenerate on the subalgebra ${\mathfrak{h}}$.  
We write ${\mathfrak{g}}={\mathfrak{h}}+{\mathfrak{h}}^{\perp}$
 for direct sum decomposition,
 and 
$
   {\mathfrak{g}}_{\mathbb{C}}={\mathfrak{h}}_{\mathbb{C}}+{\mathfrak{h}}_{\mathbb{C}}^{\perp}
$
 for its complexification.  

 Recall that each hyperbolic element $Y \in {\mathfrak{g}}$ determines a parabolic subalgebra of ${\mathfrak{g}}$, consisting of the sum of eigenspaces of $\operatorname{ad}(Y)$ corresponding to non-negative eigenvalues.
 
\begin{dfn}[Relative Borel subalgebra for $G/H$, see \cite{xk22_JA}]
\label{def:PGH}
Let $(G,H)$ be a reductive symmetric pair.  
A {{\it{Borel subalgebra} {${\mathfrak{b}}_{G/H}$}
}}
 for $G/H$
is a parabolic subalgebra of ${\mathfrak{g}}_{\mathbb{C}}$.  
It is defined by a generic element of 
$
   {\mathfrak{h}}_{\mathbb{C}}^{\perp} \cap \sqrt{-1}{\mathfrak{g}}_U
$
 or by its conjugate under an inner automorphism of $G_{\mathbb{C}}$.  
\end{dfn}

The relative Borel subalgebra ${\mathfrak{b}}_{G/H}$ 
is not necessarily solvable, 
 and thus its Levi subalgebra ${\mathfrak{l}}_{G/H}$ 
 is not always abelian.  
We note that ${\mathfrak{b}}_{G/H}$ and ${\mathfrak{l}}_{G/H}$ are determined solely from the complexified symmetric pair $({\mathfrak{g}}_{\mathbb{C}}, {\mathfrak{h}}_{\mathbb{C}})$.

\medskip
The Levi subalgebra of the relative Borel subalgebra ${\mathfrak{b}}_{G/H}$ 
for the symmetric space 
$$
   G/H=GL(n, {\mathbb{R}})/(GL(\ell,{\mathbb{R}}) \times GL(n-\ell, {\mathbb{R}}))
$$
 is given by
\begin{equation}
\label{eqn:Levi_GH_GL}
\mathfrak{l}_{G/H}=\mathbb{C}^{2\ell} \oplus \mathfrak{gl}(n-2\ell, \mathbb{C})
\end{equation}
if $2 \ell \le n$. 

On the other hand, 
 for the group $G= U(p,q)$, 
 the symmetric spaces 
\[
U(p,q)/(U(r,s) \times U(p-r,q-s)),
\quad
\text{for $2 r \le p$ and $2s \le q$}
\]
are not isomorphic to each other for different $(r,s)$.
However, they share the same complex Levi subalgebra
 as long as $r+s$ is constant (say $=\ell$).
 The corresponding complex Levi subalgebras are also isomorphic to the complex Levi subalgebra (\ref{eqn:Levi_GH_GL}) of the symmetric space 
 \[
 GL(n,\mathbb{R})/(GL(\ell,\mathbb{R}) \times GL(n-\ell,\mathbb{R})).
 \]
In contrast, the real Levi subgroups  $L_{p,q; r, s}^U$ that appear in cohomological parabolic induction are different.
For the symmetric spaces 
 $U(p,q)/(U(r,s) \times U(p-r,q-s))$,
 the real Levi subgroup is given by
\begin{equation}
\label{eqn:L_pqrs}
 L_{p,q; r, s}^U=
 {\mathbb{T}}^{2r+2s}
 \times U(p-2r, q-2s),
\end{equation}
 whereas for the symmetric space $GL(n,{\mathbb{R}})/(GL(\ell,{\mathbb{R}}) \times GL(n-\ell,{\mathbb{R}}))$,
 the corresponding real Levi subgroup is 
 \[
L_{n;p}^{\mathbb{R}}=
 ({\mathbb{C}}^{\times})^{\ell} \times GL(n-2 \ell,{\mathbb{R}}).
 \]
See also \cite{MR23} for further examples in more details.

\medskip
\begin{remark}

\begin{enumerate}
\item The real Levi subgroups of a symmetric space are  Levi subgroups of $\theta$-stable parabolic subgroups which were used to obtain the representations in the discrete spectrum of the symmetric space via cohomological induction.
\item Observe that even in the rank one case, the non-compact symmetric spaces 
\[L^U_{p,q;1,0}/L^U_{p,q;1,0}\cap (U(1,0) \times U(p-1,q)) 
\]
 and 
 \[
 L^{\mathbb{R}}_{n,1}/L^{\mathbb{R}}_{n,1}\cap (GL(1,\mathbb{R}) \times GL(n-1,\mathbb{R}))
 \]
 of the real Levi subgroups are not isomorphic.
 On the other hand, they have the same complex Levi subalgebra given in Definition~\ref{def:PGH}.
 \item \  Observe that the complex Levi subalgebras of the symmetric spaces $U(2n,2n)/(U(n,n)\times U(n,n))$
 and $U(2n,2n)/GL(2n,\mathbb{C})$ are isomorphic.
 \end{enumerate}
 
\end{remark}

\medskip
\subsection{Arthur Packets and Discrete Series Representations for Symmetric Spaces.} 

\medskip
 We recall some results about Arthur packets and representations in the discrete spectrum for the symmetric spaces.

 \medskip

 Given a fixed $\ell $ so that $2 \ell \le n$ and non-singular integral infinitesimal character \eqref{eqn:GLinflmd},  
  C.\ Moeglin and D.\ Renard showed in \cite{MR23} that all the representations with this infinitesimal  character, 
 which are in the discrete spectrum of
$GL(n, {\mathbb{R}})/(GL(\ell,{\mathbb{R}}) \times GL(n-\ell, {\mathbb{R}}))$, are in the same Arthur packet $\mathcal{A}(\lambda)$. 
However, Arthur packets for $GL(n,{\mathbb{R}})$  contains only one representation \cite{AAM}.  
Thus 
 the irreducible unitary representation $\Pi_\ell(\lambda) $
in Proposition \ref{prop:DiscGH}, 
 satisfying  the regularity condition 
 $
   \lambda_1 > \cdots > \lambda_{\ell-1} >\lambda_{\ell} > \max(n-2\ell-1, 0),
$ 
 is the only representation in the Arthur packet with this infinitesimal character.
In contrast the unitary representations in the discrete spectrum of the symmetric spaces
 $$U(p,q)/(U(1,0) \times U(p-1,q)) \text{ and } U(p,q)/(U(0,1)\times U(p,q-1)),$$
 which have the same non-singular infinitesimal character, 
 are not isomorphic, but they are in the same Arthur packet \cite{OS25}. 
 More surprisingly for fixed non-singular integral infinitesimal character the representations in discrete spectrum of the symmetric spaces $$U(2n,2n)/(U(n,n)\times U(n,n))
 \text{ and } U(2n,2n)/GL(2n,\mathbb{C})$$ are in the same Arthur packet, although the  symmetric spaces are not isomorphic.  
Thus any generalization of the GGP conjecture to unitary symmetric spaces for $p+q \equiv 0 \mod 4$ has to take this into account.

 Using the observation that Arthur packets for $GL(n,\mathbb{R})$ contain exactly one representation, these ideas lead to the reformulation of the conclusion of Theorem \ref{thm:12} as follows.

Let $\Pi$ and $\pi$ be discrete series representations in 
\[
L^2(GL(n, {\mathbb{R}})/(GL(\ell,{\mathbb{R}}) \times GL(n-\ell, {\mathbb{R}}))),
\]
and 
\[
 L^2(GL(n-1, {\mathbb{R}})/(GL(\ell,{\mathbb{R}}) \times GL(n-\ell-1, {\mathbb{R}}))), 
\] 
respectively, where $2 \ell \le n-1$.
We also assume that they have non-singular infinitesimal characters.
 Let $\mathcal{A}_{\Pi}$ and $\mathcal{A}_\pi$ be Arthur packets, such that 
\[
   \text{$\Pi \in \mathcal{A}_{\Pi}$\quad and \quad$\pi \in \mathcal{A}_{\pi}$.}  
 \]
We can summarize our discussion as follows.

\begin{corollary}
Under the above assumptions 
 we have:
\[  \operatorname{Hom}_{G'}(\Pi|_{G'},\pi)
      =\mathbb{C} 
\]
for all pairs of representations $\Pi \in \mathcal{A}_{\Pi}$ and $\pi \in \mathcal{A}_{\pi}$.  
\end{corollary}

\bigskip
\subsection{\lq\lq{Operations on the Unitary Dual}\rq\rq.} 

In the article \cite{Ven05}, 
 A.\ Venkatesh  discusses the restriction of representations 
 $\Pi_{\ell}(\lambda )$ of $GL(n,\mathbb{R})$
 to a subgroup $GL(n-1,\mathbb{R})$ embedded in the upper left corner as the stabilizer of the last coordinate vector.
More generally in this paper, he discusses for $GL(n)$, the effect on the unitary dual of the following operations: restriction to a Levi subgroup, induction
from Levi subgroups and tensor products. 
 Without explicitly computing symmetry breaking operators or referring to symmetric spaces, using only the Mackey machine, 
 A.\ Venkatesh  considers representations induced from a representation of $GL(p,\mathbb{R}) \times GL(n-p,\mathbb{R})$ to $GL(n,\mathbb{R})$ which is tempered on $GL(p,\mathbb{R})$ and trivial on $GL(n-p,\mathbb{R})$ and their restriction to the subgroup $GL(n-1,\mathbb{R})$ proving conjectures by L.\ Clozel about the automorphic support of the restriction.
We quote from the abstract of the article by L.\ Clozel \cite{Clo}:
\lq\lq{The Burger--Sarnak principle states that the restriction to a reductive subgroup of an automorphic representation of a reductive group has automorphic support. 
Arthur's conjectures parametrize automorphic representations by means of the (Langlands) dual group. 
Taken together, these principles, combined with some new arguments, imply that unipotent orbits in a Langlands dual behave functorially with respect to arbitrary morphisms $H \to G$
 of semisimple groups. 
The existence of this functoriality is proven for $SL(n)$, and combinatorial descriptions of it (due to Kazhdan, Venkatesh, and Waldspurger)
 are proposed.\rq \rq  
 
\medskip
In this article, we used different techniques to analyze the restriction of a family of non-tempered unitary representations of $GL(n,\mathbb{R})$ to $GL(n-1, \mathbb{R})$ and proved the existence of non-trivial SBOs.
We also provided a proof for symmetry breaking for some tempered representations.

\medskip
\subsection{A GGP Theorem for the Symmetric Spaces \linebreak  $GL(n,\mathbb{R})/GL(p,\mathbb{R})\times GL(n-p,\mathbb{R})$ and $U(p,q)/U(r,s)\times U(p-r,q-s)$? }

Around 1992, B.\ Gross and D.\ Prasad published conjectures concerning the restriction of discrete series representations of orthogonal groups to smaller orthogonal groups \cite{GP92}. 
These have been generalized to unitary groups and have been proven by H.\ He \cite{He} for individual discrete series representations.

 
These ideas can be generalized in two directions: 
\begin{itemize}
    \item understand the symmetry breaking of discrete series of symmetric spaces and discrete series of subspaces,
    \item  understand the symmetry breaking  of representations in Arthur packets of groups and of subgroups.
\end{itemize}

Discrete series representations of a symmetric space $G/H$ are generally not tempered representations.
See \cite{BK21} for the classification of $G/H$ such that the regular representation on $L^2(G/H)$ is non-tempered. In \cite{MR23} D.~Renard and C.~Moeglin examine the relationship between Arthur packets and discrete series representations of symmetric spaces of classical groups.
The representations in the discrete spectrum of symmetric spaces are members of an Arthur packet \cite{MR23}. 
Not all members of such an Arthur packet are discrete series representations of a symmetric space, and an Arthur packet may contain discrete series representations of several symmetric spaces \cite{MR23}.  
Generalizing the GGP conjectures to symmetric spaces involves generalizing them to the subset of representations in a given Arthur packet which are discrete series of a symmetric space.  
Moeglin and Renard showed for real classical groups that if a representation in an Arthur packet is in the discrete spectrum of a symmetric space, 
then another representation in the same packet is either in the discrete spectrum of no symmetric space or in the discrete spectrum of a unique symmetric space \cite{MR23}.  
On the other hand the results in \cite{OS25} suggest that it may be possible to generalize the GGP conjectures to discrete series representations of symmetric spaces or Stiefel manifolds.

\medskip

\subsection{Examples That Illustrate Our Results} 
\subsubsection{Branching of Limit of Discrete Series}

By Theorems~\ref{thm:23081404}  and~\ref{thm:24102601} we can deduce non-vanishing results concerning symmetry breaking for some limits of discrete series representations
from those for discrete series representations.
As an illustration, we consider the pair
\[
(G,G')=(U(2,1), U(1,1)).
\]
We use the conventions from Sections \ref{sec:upq} and \ref{sec:upq2}.

The group $G=U(2,1)$ has three families of discrete series representations. 
They are parametrized by the Harish-Chandra parameters $(x_1,x_2,y) \in (\mathbb{Z})^{2}_{>}\times \mathbb{Z}$:
\begin{equation}
\label{eqn:u213}
x_1>x_2>y,  \qquad  x_1>y>x_2,  \qquad y>x_1>x_2,
\end{equation}
which we also denote symbolically, following \cite{He}, by
\[\ 
+ +  -,  \qquad  + - +,  \qquad - + +.
\]

The group $G'=U(1,1)$ has two families of discrete series representations.
They are  parametrized by  Harish-Chandra parameters $(\xi, \eta) \in (\mathbb{Z}+\tfrac{1}{2})^2$:
\[ \xi > \eta, \qquad \textrm{and} \qquad \eta>\xi, \] 
which are denoted symbolically by
$\op \oi$ and $\oi \op$, respectively.

The $\tau$-invariants (Definition~\ref{def:tau_inv}) of the representations $\Pi_{(x_1, x_2, y)}$, $\Pi_{(x_1,y, x_2)}$, $\Pi_{(y, x_1, x_2)}$ of $G=U(2,1)$ corresponding to \eqref{eqn:u213} are $e_1-e_2$, none, and $e_2-e_3$, respectively.
Applying the vanishing theorem
(Theorem~\ref{thm:25080811}),
we obtain that 
\[
\operatorname{Hom}_{G'}\big(\Pi_{(x_1,x_2,y)}|_{G'}, \pi_{(\xi,\eta)}\big)= 0
\]
whenever, for instance, $x_1 > x_2> \xi > y>\eta$
or
$\xi>x_1 > x_2 > y>\eta$.

By Theorem~\ref{thm:25080811}),
$\operatorname{Hom}_{G'}(\Pi_{(x_1,x_2,y)}|_{G'}, \pi_{(\xi,\eta)})$ vanishes
in a total of 24 interleaving patterns.

An explicit condition on
\[(x_1, x_2, y; \xi, \eta) \in \mathbb Z^3 \times \left(\mathbb Z+\tfrac{1}{2}\right)^2
\]
such that 
\[
\operatorname{Hom}_{G'}\big(\Pi_{(x_1,x_2,y)}|_{G'}, \pi_{(\xi,\eta)}\big)\neq 0
\]
is given in the left column of Table~\ref{tab:U21},
as shown by H.~He \cite{He}.
These conditions are described by
\emph{interleaving patterns}, for which an intrinsic explanation was provided in Theorem~\ref{thm:24102706} from a different perspective.

\begin{table}[H]
\begin{center}
\begin{tabular}{c|c|c|c|c}
\text{Case}
&\text{interleaving}
&$\delta$
&$\delta'$
&\text{limit of discrete series}
\\
\hline
\text{I}
&
$x_1 > \xi > x_2 > y>\eta$
&$++-$
&$\op\oi$
&$ x_1>\xi>x_2 = y>\eta$
\\[.3em]
\text{II}
&
$x_1>\xi>\eta>y>x_2$
&$+-+$
&$\op\oi$
&$x_1>\xi>\eta>y=x_2$
\\[.3em]
\text{III}
&
$x_1>y>x_2>\xi>\eta$
&$+-+$
&$\op\oi$
&$x_1=y>x_2>\xi>\eta$
\\[.3em]
&
&
&
&$x_1>y=x_2>\xi>\eta$
\\[.3em]
\text{IV}
&
$x_1>y>\eta>\xi>x_2$
&$+-+$
&$\oi\op$
&$x_1=y>\eta>\xi>x_2$
\\[.3em]
\text{V}
&
$\eta>\xi>x_1>y>x_2$
&$+-+$
&$\oi\op$
&$\eta>\xi>x_1=y>x_2$
\\[.3em]
&
&
&
&$\eta>\xi>x_1>y=x_2$
\\[.3em]
\text{VI}
&
$\eta>y>x_1>\xi>x_2$
&$-++$
&$\oi\op$
&$\eta>y=x_1>\xi>x_2$
\\
\end{tabular}
\end{center}
\caption{Interleaving patterns for $(U(2,1), U(1,1))$}
\label{tab:U21}
\end{table}

For instance, fix $\xi > \eta$, and consider discrete series representation $\pi_{(\xi,\eta)} $ of $G'=U(1,1)$.
 The interleaving patterns corresponding to discrete series representations $\Pi_{(x_1, x_2, y)} $ of $G=U(2,1)$ satisfying 
\[
\operatorname{Hom}_{G'}(\Pi_{(x_1, x_2, y)}|_{G'},\pi_{(\xi, \eta)}) \ne 0
\]
are Cases~I, II, and~III. 
Among these three, only interleaving pattern~II is \emph{coherent} in the sense of \cite[Def.~4.5]{HKS}.

The nondegenerate limits of discrete series representations corresponding to \eqref{eqn:u213} 
are given by 
$x_2=y$, $x_1=y$ or $y=x_2$, and $y=x_1$, respectively. 
They occur as a direct summand of certain reducible principal series representations.
Applying once again Theorems~\ref{thm:23081404} and~\ref{thm:24102601} on translation functors for symmetry breaking,
we deduce that there exists a nonzero symmetry breaking operator between the limit of a generic discrete series representation and a discrete series representation of $U(1,1)$.
This application to limits of discrete series representations is  new.
For example, in the interleaving pattern~III, the cases $x_1=y$ or $x_2=y$ are allowed.
A similar argument applies to the other interleaving patterns~I–VI.
These conditions are summarized in the right column of Table~\ref{tab:U21}.

\subsubsection{Tempered Case ($q=2s$)}
We now illustrate Theorem~\ref{thm:250126FJ} 
with an example
\[
(G, G')=(U(3,2), U(2,2))
\]
with 
\[(H,H')=(U(1,1)\times U(2,1), U(1,1)\times U(1,1)).
\]
There are, up to conjugation, two $\theta$-stable parabolic subalgebras of 
$\mathfrak{g}_\mathbb{C} = \mathfrak{gl}(5, \mathbb{C})$ that describe discrete series representations for $G/H$. 
They are parametrized, as in Lemma~\ref{lem:250129}, 
by interleaving patterns 
$D \in {\mathfrak{P}}({\mathbb{R}}^{r,s})$ in ${\mathbb{R}}_{>}^r \times {\mathbb{R}}_{>}^s$, 
or equivalently by the data $\kappa = \{(r_j), (s_j), M\}$ with $1 \le M \le \min(r,s)$, 
where $(r,s) = (1,1)$:
\[
\begin{array}{l l l}
\text{Case 1:} & D^{(1)} = \{ x > y > 0 \}, & \text{that is, } \kappa^{(1)} = \{(1),(1),1\}, \\
\text{Case 2:} & D^{(2)} = \{ y > x > 0 \}, & \text{that is, } \kappa^{(2)} = \{(0),(1),1\}.
\end{array}
\]
Accordingly, the set of discrete series for $G/H$ is given by
\[
\operatorname{Disc}(G/H) = \{\Pi_{(x,y)} : (x,y) \in \mathbb{Z}^2 \cap (D^{(1)} \cup D^{(2)}) \}.
\]
The representation $\Pi_{(x,y)}$ has a 
$\mathfrak{Z}(\mathfrak{g}_\mathbb{C})$-infinitesimal character 
\[(x,0,-x,y,-y) \in \mathbb C^5/\mathfrak S_5,
\]
while its minimal $K$-type has highest weight

\[
\mu_\lambda =
\begin{cases}
(x, 0, -x, y, -y), & \text{in Case 1}, \\
(x-1, 0, 1-x, y+1, -y-1), & \text{in Case 2}.
\end{cases}
\]

Similarly, up to conjugation, there are two $\theta$-stable parabolic subalgebras of 
$\mathfrak{g}'_\mathbb{C} = \mathfrak{gl}(4, \mathbb{C})$ 
that describe discrete series representations for $G/H$. 
They are parametrized as
\[
\begin{array}{l l l}
\text{Case 1:} & D'^{(1)} = \{ \xi > \eta > 0 \}, & \text{that is, } \kappa'^{(1)} = \{(1),(1),1\}, \\
\text{Case 2:} & D'^{(2)} = \{ \eta > \xi > 0 \}, & \text{that is, } \kappa'^{(2)} = \{(0),(1),1\}.
\end{array}
\]
Thus, the set of discrete series for $G'/H'$ is given by
\[
\operatorname{Disc}(G'/H') = \{\pi_{(\xi,\eta)} : (\xi,\eta) \in (\mathbb{Z}+\tfrac{1}{2})^2 \cap (D'^{(1)} \cup D'^{(2)}) \}.
\]

The representation $\pi_{(\xi,\eta)}$  has a 
$\mathfrak{Z}(\mathfrak{g}'_\mathbb{C})$-infinitesimal character: 
\[(\xi,-\xi,\eta,-\eta) \in \mathbb C^4/\mathfrak S_4,
\]
while its minimal $K$-type has highest weight
\[
\mu_\lambda =
\begin{cases}
\begin{aligned}
& \xi + \tfrac{1}{2}, & & -\xi - \tfrac{1}{2}, & & \eta - \tfrac{1}{2}, & & \tfrac{1}{2} - \eta 
&& \text{(Case 1)}, \\
& \xi - \tfrac{1}{2}, & & \tfrac{1}{2} - \xi, & & \eta + \tfrac{1}{2}, & & -\eta - \tfrac{1}{2} 
&& \text{(Case 2)}.
\end{aligned}
\end{cases}
\]

By the non-vanishing theorem of period integrals (Theorem~\ref{thm:24012120}), 
\[
\operatorname{Hom}_{G'}(\Pi_{(x,y)}|_{G'}, \pi_{(\xi, \eta)}) \neq 0
\]
if
\[
\begin{cases}
\begin{aligned}
x &= \xi + \tfrac{1}{2}, & y &= \eta - \tfrac{1}{2} & \text{(Case 1)}, \\
x - 1 &= \xi + \tfrac{1}{2}, & y + 1 &= \eta + \tfrac{1}{2} & \text{(Case 2)}.
\end{aligned}
\end{cases}
\]

Thus, Theorem~\ref{thm:24120703} guarantees
\[
\begin{cases}
x > \xi > \eta > y > 0 & \text{(Case 1)}, \\
\eta > y > x > \xi > 0 & \text{(Case 2)}.
\end{cases}
\]

These are the ingredients of Theorem~\ref{thm:250126FJ}.

\begin{remark}
In this specific example, all the discrete series representations are Harish-Chandra's discrete series. 
The signatures of the interleaving patterns are 
\[
+ \op \oi - + - \oi \op + \quad \text{(Case 1)}, \quad 
\oi - + \op + \op + - \oi \quad \text{(Case 2)};
\]
none of them forms a \emph{coherent pair} in the sense of \cite{HKS}. See \cite[Def.\ 3.1]{HKS} for the definition of a \emph{signature} of a pair of representations.
\end{remark}
More generally, when $q=2s$, discrete series representations for $G/H=U(p,q)/(U(r,s) \times U(p-2r, q-2s))$ are tempered for generic parameters. It is worth noting that the non-vanishing results obtained in Theorem~\ref{thm:250126FJ} hold even when the interleaving patterns do not form a coherent pair in the sense of \cite[Def.~4.5]{HKS}.

\subsubsection{Non-Tempered Case}
\medskip \noindent
Let 
\[
(G, H) = (U(p,q), U(p-1,q)).
\]
We begin with the rank-one symmetric spaces $G/H$ and $G'/H'$, where 
\[
 (H, H') = (U(1) \times U(p-1,q), U(1) \times U(p-2,q)). 
\]
Then the real Levi subgroups for the symmetric spaces $G/H$ and $G'/H'$, as given in \eqref{eqn:L_pqrs}, are
\[
 L_{p,q; 1, 0}^U= {\mathbb{T}}^{2} \times U(p-2, q),
 \quad
 L_{p-1,q; 1, 0}^U= {\mathbb{T}}^{2} \times U(p-3, q),
 \]
 respectively.
In this rank-one setting, there is only one $\theta$-stable parabolic subalgebra up to conjugation.
For $x \in \mathbb{N}+\frac{p+q-1}{2}$
and $\xi\in \mathbb{N}+\frac{p+q-2}{2}$,
 Theorem~\ref{thm:250126FJ} guarantees that
 \[
\operatorname{Hom}_{G'}(\Pi_{x}|_{G'}, \pi_{\xi}) \neq 0
\]
whenever $x>\xi (>0)$.
Note that $\Pi_x$ is non-tempered if $p \ge 3$.
This rank-one case is also discussed in \cite{OS25}.

Moreover, the stability theorem of multiplicities 
within \emph{fences} (Theorem~\ref{thm:24120703})  allows us to extend the non-vanishing result from $\Pi_{x}$ to the representations obtained by cohomological induction from a character $(x_1, x_2)$ of  $L_{p,q; 1, 0}^U$, as long as $x_1, -x_2 > \xi$.
(The representation $\Pi_x$ corresponds to the case $x_1=-x_2=x$.)
These representations do not appear as discrete series representations for the symmetric space $G/H$ when $x_1+x_2 \neq 0$, but rather as discrete series for the indefinite Stiefel manifolds $U(p,q)/U(p-1,q)$ which are not symmetric spaces;
see the case $r = s = 1$ in \cite[Chap.~2, Sect\.~3]{Kobayashi92} for further details.

\medskip
\noindent
We now consider the higher-rank, non-tempered case where 
\[
(H, H')
=(U(r,s) \times U(p-r, q-s),  U(r,s) \times U(p-1-r, q-s)),
\]
with $2r \le p-1$ and $2s \le q$.
In this setting, we have 
\[
\operatorname{rank}G/H = \operatorname{rank}G'/H' = r+s.
\]
There are $\binom{r+s}{r}$ $\theta$-stable parabolic subalgebras with real Levi subgroup
\[
L_{p,q; r, s}^U= {\mathbb{T}}^{2r+2s} \times U(p-2r, q-2s),
\]
(see \eqref{eqn:L_pqrs}) up to conjugation.
The corresponding cohomologically induced representations
are non-tempered when $q>2s$, yet Theorem~\ref{thm:250126FJ} still guarantees non-vanishing of the corresponding symmetry breaking operators.

\end{document}